\documentclass[a4paper]{article}

\usepackage[margin=1in]{geometry} % full-width

\usepackage{amsmath}
\usepackage{amsthm}
\usepackage{amssymb}
\usepackage[dvipsnames]{xcolor}
\usepackage{comment}
\usepackage{graphicx}
\usepackage{multirow}
\usepackage{mathtools}

\mathtoolsset{showonlyrefs}

\usepackage[utf8]{inputenc}
\usepackage{hyperref}
\hypersetup{
	unicode,
	pdfauthor={Lidia Fernández, Juan Antonio Villegas},
	pdftitle={A method for generating multivariate multiple orthogonal polynomials},
}

\usepackage[sort&compress,numbers,square]{natbib}

\theoremstyle{plain}
\newtheorem{theorem}{Theorem}
\newtheorem{corollary}[theorem]{Corollary}

\newtheorem{proposition}[theorem]{Proposition}

\theoremstyle{definition}
\newtheorem{definition}[theorem]{Definition}

\newtheorem{remark}[theorem]{Remark}

\newcommand{\NN}{\mathbb{N}}
\newcommand{\ZZ}{\mathbb{Z}}
\newcommand{\RR}{\mathbb{R}}

\usepackage[colorinlistoftodos,prependcaption,textsize=small]{todonotes}

\definecolor{greenpie}{rgb}{0.69, 0.95, 0.76}

\usepackage{mathrsfs} % for \mathscr command

\title{A method for generating multivariate multiple orthogonal polynomials}
\author{Lidia Fernández$^\star$ \and Juan Antonio Villegas$^\star$}

\date{
    \texttt{lidiafr@ugr.es},\qquad \texttt{jantoniovr@ugr.es}\\[2ex]%
	$^\star$Departamento de Matemática Aplicada and Instituto de Matemáticas (IMAG). Universidad de Granada \\[2ex]%
	\today
}

\begin{document}
	\maketitle
	
	\begin{abstract}

    In the framework of multiple orthogonal polynomials (MOPs) and their extension to the multivariate setting, a unified mechanism for generating structured families of multiple orthogonal polynomials across diverse multidimensional domains is still lacking. In this work, we bridge this gap by developing an extended Koornwinder-type methodology for constructing multivariate multiple orthogonal polynomials. We introduce two distinct coupling schemes: combining a univariate MOPs family with a standard orthogonal family, which enables the definition of both Type~I and Type~II bivariate systems along with their dual biorthogonality relations, and coupling two univariate MOPs families to form Type~II multivariate systems.
    
    Using this general framework, we provide the first explicit formulations of multiple orthogonal polynomials on a variety of bivariate regions, including bounded domains such as a parabolic domain and the square $[0,1]^2$, non-standard unbounded geometries like the positive quadrant $(\mathbb{R}_0^+)^2$ and the wedge $\mathbb{V}^2$, as well as new constructions on the triangle $T$. Furthermore, explicit multiple orthogonal systems on the $d$-dimensional simplex $T^d$ and on the $d$-dimensional cone $\mathbb{V}^d$ are given.\medskip
		
	\noindent\textbf{Keywords:} multiple orthogonal polynomials, multivariate orthogonal polynomials, multivariate multiple orthogonality, Koornwinder's construction, planar domains, simplex, cone.

        \noindent\textbf{MSC: } 33C45, 33C50, 42C05.
        
        \noindent\textbf{Corresponding author: } Juan Antonio Villegas.
        
	\end{abstract}

	%\tableofcontents

    \section{Introduction}\label{sec:intro}

    Multiple orthogonal polynomials (MOPs) arise when the classical notion of orthogonality is extended to settings involving several measures instead of a single one \cite{Apt98, Ism05, VA20}. In this framework, a polynomial satisfies a collection of orthogonality conditions with respect to different measures, organized according to a multi-index. This naturally leads to Type~I and Type~II formulations, which satisfy a biorthogonality relation and extend many classical results from the theory of Orthogonal Polynomials (OPs), see Section~\ref{sec:preliminaries} or \cite[Chapter~23]{Ism05}, \cite[Section~3]{VA20} for a more detailed introduction. \medskip
    
    While univariate multiple orthogonality and classical multivariate orthogonality \cite{DX14, Koo75, Lar59, Sue99} have long been subjects of active investigation, extensions of multiple orthogonality to other kinds of domains remain a comparatively young field. For instance, in \cite{CDO15, KV24a, KV24b, MVA08}, MOPs  on the unit circle and their properties are addressed. Regarding the extensions to the multivariate setting, early foundations were established in \cite{FV26a}, where a general framework for bivariate MOPs  on the whole lattice was introduced, followed by recent advances along the step-line \cite{MRW27}. Besides, Rodrigues-type formulas on the simplex were developed, along with hypergeometric expressions and applications to Hermite--Padé-type problems in \cite{FFMV26}. Furthermore, definitions of MOPs  on the unit ball and radial domains were addressed in \cite{FV26b}.\medskip

    To date, several methodologies and definitions of MOPs  on the most explored domains in bivariate standard OP theory have been given. However, despite the recent progress in multivariate multiple orthogonality, there is still a lack of a general yet simple mechanism to generate structured MOPs families across diverse multidimensional domains.
    \medskip

    The principal contribution of this work is the development of a new framework for constructing multivariate MOPs  via an extension of the techniques described by  Koornwinder in \cite{Koo75}.    
    In particular, we present two methodologies for generating MOPs  on planar domains by coupling two univariate families of (multiple) orthogonal polynomials. Using these methods, we construct several families of multiple orthogonal polynomials on planar domains (see Figure~\ref{fig:domains}). In particular, we provide the first explicit constructions of MOPs on a parabolic domain, in the positive quadrant, and on a wedge. Furthermore, we provide novel constructions of MOPs  on the triangle and on the $d$-dimensional simplex which build upon the earlier definitions in \cite{FFMV26}, and introduce the first definition of multiple orthogonal polynomials on the $d$-dimensional cone. \medskip

    The structure of the manuscript is as follows: the preliminary tools, such as an introduction to multiple orthogonal polynomials, some classical families of both standard and multiple orthogonal polynomials, and an explanation of Koornwinder's construction for bivariate OPs are presented in Section~\ref{sec:preliminaries}. In Section~\ref{sec:koornwinder}, Koornwinder's construction is extended to the multiple setting, providing two different ways of combining univariate orthogonal polynomials, together with their respective advantages, limitations and main results of both approaches. Next, in Section~\ref{sec:examples} we employ the methods introduced above to give explicit constructions for MOPs  on a parabolic domain, the square $[0,1]^2$, the positive quadrant $(\mathbb R_0^+)^2$, the wedge $\mathbb V^2$ and the triangle $T$. Finally, in Sections~\ref{sec:MOP-simplex}~and~\ref{sec:MOP-cone}, leveraging existing extensions of Koornwinder's method, we provide MOPs  on the $d$-variate simplex $T^d$ and the cone $\mathbb V^d$, respectively. The manuscript concludes with the main conclusions and open problems in Section~\ref{sec:conclusions}.

    \section{Preliminaries}
    \label{sec:preliminaries}
    
    Before introducing the main novel contributions of this work, in this section we collect the necessary background material, fundamental definitions, and key properties of multiple orthogonal polynomials  in one variable, as well as bivariate orthogonal systems generated via Koornwinder's classical construction.
    
    \subsection{Multiple Orthogonal Polynomials}
    \label{subsec:mops_general}
    
    Multiple orthogonal polynomials (MOPs) are a natural extension of standard orthogonal polynomials where the orthogonality conditions are distributed over a vector of $r \in\NN$ weight functions \cite{Apt98, Ism05, MFVA16, VA06, VA20}. Throughout this paper, let $r$ denote the number of weights, and let $\vec{n} = (n_1, \dots, n_r) \in \mathbb{N}_0^r$ be a multi-index of total length $|\vec{n}| := n_1+\cdots+n_r$.
    
    Let $\{\omega_1, \dots, \omega_r\}$ be a collection of positive weight functions supported on $I_1, \dots, I_r \subseteq \mathbb{R}$, respectively, such that all moments
    \begin{equation}
    \int_{I_j} x^k \omega_j(x) \, \mathrm dx, \quad k \in \mathbb{N}_0, \quad j = 1, \dots, r,
    \end{equation}
    are finite.
    
    There are two primary types of multiple orthogonal polynomials associated with the multi-index $\vec{n}$ and the system of weights $\{\omega_1, \dots, \omega_r\}$:
    
    \begin{definition}[Type~I Multiple Orthogonal Polynomials]
    \label{def:typeI_mops}
    The Type~I multiple orthogonal polynomials associated with $\vec{n}$ and $\{\omega_1, \dots, \omega_r\}$ are $r$ polynomials $A_{\vec{n},1}, \dots, A_{\vec{n},r}$, where $\deg A_{\vec{n},j} \le n_j - 1$ for each $j = 1, \dots, r$, not all identically zero, satisfying the orthogonality relations
    \begin{equation}
    \label{eq:typeI_orthogonality}
    \sum_{j=1}^r \int_{I_j} A_{\vec{n},j}(x)\, x^k \, \omega_j(x) \, \mathrm dx = 0, \quad k = 0, 1, \dots, |\vec{n}| - 2,
    \end{equation}
    subject to a normalization condition, typically given by
    \begin{equation}
    \label{eq:typeI_normalization}
    \sum_{j=1}^r \int_{I_j} A_{\vec{n},j}(x)\,  x^{|\vec{n}|-1}\, \omega_j(x) \, \mathrm dx = 1.
    \end{equation}
    \end{definition}
    
    \begin{definition}[Type~II Multiple Orthogonal Polynomials]
    \label{def:typeII_mops}
    The Type~II multiple orthogonal polynomial associated with $\vec{n}$ and $\{\omega_1, \dots, \omega_r\}$ is the monic polynomial $P_{\vec{n}}(x)$ of degree $|\vec{n}|$ satisfying the orthogonality conditions
    \begin{equation}
    \label{eq:typeII_orthogonality}
    \int_{I_j} P_{\vec{n}}(x)\, x^k \, \omega_j(x) \, \mathrm dx = 0, \quad k = 0, 1, \dots, n_j - 1, \quad j = 1, \dots, r.
    \end{equation}
    \end{definition}
    
    The existence and uniqueness of Type~I and Type~II polynomials depend directly on the structural properties of the weight system.
    
    \begin{definition}[Normal Multi-Indices and Perfect Systems]
    \label{def:normal_perfect}
    A multi-index $\vec{n} \in \mathbb{N}_0^r$ is called \emph{normal} with respect to the system of weights $\{\omega_1, \dots, \omega_r\}$ if the monic Type~II polynomial $P_{\vec{n}}(x)$ exists uniquely and has degree exactly $|\vec{n}|$ (or, equivalently, if the Type~I polynomials $A_{\vec{n},1}, \dots, A_{\vec{n},r}$ are uniquely determined up to a scalar factor).  
    
    If every multi-index $\vec{n} \in \mathbb{N}_0^r$ is normal, the system of weights $\{\omega_1, \dots, \omega_r\}$ is said to be \emph{perfect}.
    \end{definition}
    
    A prominent class of perfect systems consists of \emph{Algebraic Tchebycheff systems} (\emph{AT-systems}). The weights $\{\omega_1, \dots, \omega_r\}$ form an AT-system on a common interval $I$ if for every multi-index $\vec{n} \in \mathbb{N}_0^r$, any non-zero linear combination of the form
    \begin{equation}
    \sum_{j=1}^r q_{j}(x) \omega_j(x), \quad \text{with } \deg q_{j} \le n_j - 1,
    \end{equation}
    has at most $|\vec{n}| - 1$ zeros in $I$. Any AT-system is perfect \cite[Theorem~23.1.4]{Ism05}, ensuring full regularity for all multi-indices. Other classic examples of perfect systems include Angelesco systems and Nikishin systems \cite{ABVA03, Apt98, NS91, VA06}.
    
    Indeed, the concept of perfect system requires that the associated weights be ``sufficiently different'' from one another. A family of non-perfect systems is given in the following result.
    \begin{proposition}\label{prop:MOP-symmetric}
        Consider a system of positive weights $\omega_1,\dots,\omega_r$, all of them even and supported on a symmetric interval $[-a,a]$. Then, the associated system of measures is not perfect.
    \end{proposition}
    \begin{proof}
        To prove that a system is not perfect, it suffices to find a non-normal multi-index. For example, consider the multi-index $\vec n =(1,1,0,\dots,0)\in\NN^r_0$.

        Its associated matrix is
        
        $$
        M_{\vec n}=\begin{pmatrix}
            m_{0}^{(1)} & m_{1}^{(1)} \\
            m_{0}^{(2)} & m_{1}^{(2)} 
        \end{pmatrix},
        $$
        where $m_k^{(j)}$, $k\geq 0$, $1\leq j\leq r$ are the moments of the $j$-th measure. Since $\omega_j$ is an even function supported on a symmetric interval,  $m_{2k+1}^{(j)}=\int_{-a}^a x^{2k+1} \omega_j(x) \,\mathrm  dx=0$ for every $k$, and in particular $m_1^{(1)}=m_1^{(2)}=0$. Consequently, $\det(M_{\vec n})=0$ and $\vec n$ is not a normal multi-index.
    \end{proof}

    Type~I and Type~II multiple orthogonal polynomials are linked by a fundamental biorthogonality relation. Defining the Type~I function
    \begin{equation}
    \label{eq:typeI_function}
    Q_{\vec{n}}(x) := \sum_{j=1}^r A_{\vec{n},j}(x) \omega_j(x),
    \end{equation}
    and assuming the normalization \eqref{eq:typeI_normalization}, the duality between Type~I and Type~II systems is expressed as \cite[Theorem~23.1.6]{Ism05}, \cite{Apt98, VA20}:
    \begin{equation}
    \label{eq:biorthogonality}
     \int_{\RR} P_{\vec{n}}(x) Q_{\vec{m}}(x) \, \mathrm dx = 
    \begin{cases}
    0, & \text{if } \vec m\le\vec n\text{ (componentwise)} \\
    0, & \text{if } |\vec n|\leq |\vec m|-2 \\
    1, & \text{if } |\vec n|=|\vec m|-1,
    \end{cases}
    \end{equation}
    
    Having reviewed the general framework of multiple orthogonality, we now describe some of the most representative classical families of OPs and MOPs.

    \subsection{Univariate Classical Orthogonal Polynomials: Jacobi and Gegenbauer}
    \label{subsec:jacobi_gegenbauer}
    
    Before addressing multiple orthogonality, we recall the standard univariate Jacobi and Gegenbauer polynomials, as they serve as primary building blocks in multivariate constructions such as Koornwinder's scheme \cite{AS64, Sze59}.
    
    \subsubsection{Jacobi Polynomials}
    \label{subsubsec:jacobi_standard}
    
    The classical Jacobi polynomials, denoted by $P_n^{(\alpha,\beta)}(x)$, form a sequence of orthogonal polynomials on the standard interval $[-1,1]$ with respect to the weight function
    \begin{equation}
    \label{eq:jacobi_weight_standard}
    w^{(\alpha,\beta)}(x) = (1-x)^\alpha (1+x)^\beta, \quad x \in [-1,1],
    \end{equation}
    where $\alpha, \beta > -1$.
    
    Their orthogonality relation is given by
    \begin{equation}
    \label{eq:jacobi_orth_standard}
    \int_{-1}^1 P_n^{(\alpha,\beta)}(x)\, x^k \, (1-x)^\alpha (1+x)^\beta \, \mathrm dx = 0\qquad\text{ if } 0\leq k<n.
    \end{equation}
    
    For many applications (including shifted constructions on the unit interval $[0,1]$), it is convenient to map the domain via the linear transformation $x \longmapsto 1-2t$ for $t \in [0,1]$. Under this change of variable, the shifted Jacobi polynomials $\widetilde{P}_n^{(\alpha,\beta)}(t) := P_n^{(\alpha,\beta)}(1-2t)$ satisfy the orthogonality relation on $[0,1]$:
    \begin{equation}
    \label{eq:jacobi_orth_shifted}
    \int_{0}^1 \widetilde{P}_n^{(\alpha,\beta)}(t)\, t^k \, t^\alpha(1-t)^\beta  \, \mathrm dt =  0\qquad\text{ if } 0\leq k<n.
    \end{equation}
    Explicit formulas and representations for Jacobi polynomials can be found in \cite[Chapter~22]{AS64}.
    
    \subsubsection{Gegenbauer Polynomials}
    \label{subsubsec:gegenbauer_standard}
    
    Gegenbauer (or ultraspherical) polynomials, denoted by $C_n^{(\lambda)}(x)$, arise as a symmetric special case of Jacobi polynomials corresponding to equal parameters $\alpha = \beta = \lambda - 1/2$ with $\lambda > -1/2$ ($\lambda \neq 0$).
    
    The Gegenbauer polynomials are orthogonal on $[-1,1]$ with respect to the weight function
    \begin{equation}
    \label{eq:gegenbauer_weight}
    w^{(\lambda)}(x) = (1-x^2)^{\lambda - 1/2}, \quad x \in [-1,1],
    \end{equation}
    and satisfy the orthogonality condition
    \begin{equation}
    \label{eq:gegenbauer_orth}
    \int_{-1}^1 C_n^{(\lambda)}(x)\, x^k \, (1-x^2)^{\lambda - 1/2} \, \mathrm dx =  0\qquad\text{ if } 0\leq k<n.
    \end{equation}
    Moreover, as $[-1,1]$ is a symmetric interval and $w^{(\lambda)}$ is an even weight, \eqref{eq:gegenbauer_orth} also holds whenever $k-n$ is odd.
    
    Special cases of Gegenbauer polynomials include Legendre polynomials ($\lambda = 1/2$) and Chebyshev polynomials of the first and second kinds ($\lambda \to 0$ and $\lambda = 1$, respectively), making Gegenbauer polynomials crucial in numerical integration and harmonic analysis on spheres and balls \cite{DX13, Sze59}.

    Observe that, owing to Proposition~\ref{prop:MOP-symmetric}, it is not possible to have a perfect system of ``multiple Gegenbauer polynomials'', as the weight \eqref{eq:gegenbauer_weight} is symmetric and supported on a symmetric interval.
    
    \subsection{Families of Multiple Orthogonal Polynomials}
    \label{subsec:mop_families}
    
    In this subsection, we present the classical Jacobi--Pi\~neiro and multiple Laguerre (first kind) polynomials.
    
    \subsubsection{Jacobi--Pi\~neiro Polynomials}
    \label{subsubsec:jacobi_pineiro}
    
    Jacobi--Pi\~neiro polynomials were originally introduced by Pi\~neiro \cite{Pin87} in the context of simultaneous Hermite--Pad\'e approximation for systems of hypergeometric series. They play a prominent role in number theory (e.g., in irrationality proofs for values of the Riemann zeta function \cite{Fis04}), stochastic processes, urn models, and Markov chains \cite{BDFMMAF24, GDI22}, as well as in electrostatic models and random matrix theory \cite{BK04, MFOSL23, NVA16}.
    
    They are MOPs  defined on $I = [0,1]$ that extend the shifted Jacobi polynomials $\{\widetilde P^{(\alpha,\beta)}_n\}$ introduced in Section~\ref{subsubsec:jacobi_standard} to the multiple setting. Depending on whether vector or scalar exponents are used in the Jacobi factors, three canonical configurations arise:
    
    \paragraph{Case 1: Exponents $x^{\alpha_j} (1-x)^\beta$.}
    Consider $r$ weight functions on $(0,1)$ defined by
    \begin{equation}
    \label{eq:jp_weights_alpha_vec}
    \omega_j(x) = x^{\alpha_j} (1-x)^\beta, \quad j = 1, \dots, r,
    \end{equation}
    with parameters $\alpha_j > -1$, $j = 1, \dots, r$, and $\beta > -1$.
    
    When the parameters satisfy $\alpha_j - \alpha_k \notin \mathbb{Z}$ for all $j \neq k$, the weight system $\{\omega_1, \dots, \omega_r\}$ forms an AT-system on $(0,1)$ \cite{BCVA05, BDFMM23a}. Consequently, the system is perfect, guaranteeing the existence and uniqueness of the monic Type~II Jacobi--Pi\~neiro polynomial $P_{\vec{n}}^{(\vec{\alpha},\beta)}(x)$ of degree $|\vec{n}|$ for every multi-index $\vec{n} \in \mathbb{N}_0^r$.
    
    Its Type~II orthogonality equations are explicitly given by:
    \begin{equation}
    \label{eq:jp_orth_alpha_vec}
    \int_{0}^1  P_{\vec{n}}^{(\vec{\alpha},\beta)}(x)\, x^k \, x^{\alpha_j} (1-x)^\beta \, \mathrm dx = 0, \quad k = 0, 1, \dots, n_j - 1, \quad j = 1, \dots, r.
    \end{equation}

    Type~II Jacobi--Pi\~neiro polynomials can be efficiently computed (up to normalization constant) by using the Rodrigues-type formula \cite{ABVA03}:
    \begin{equation}
        \label{eq:rodrigues-jp}
        P^{(\vec \alpha, \beta)}_{\vec n}(x) = \left(\dfrac{1}{\omega_r(x)}\dfrac{\mathrm d^{\,n_r}}{\mathrm dx^{n_r}}\omega_r(x)\right)\cdots\left(\dfrac{1}{\omega_1(x)}\dfrac{\mathrm d^{\,n_1}}{\mathrm dx^{n_1}}\omega_1(x)\right)\phi^{|\vec n|}(x),
    \end{equation}
    where $\phi(x)=x(x-1)$.
    
    \paragraph{Case 2: Exponents $x^\alpha (1-x)^{\beta_j}$.}
    Consider the weight functions on $(0,1)$ given by
    \begin{equation}
    \label{eq:jp_weights_beta_vec}
    \omega_j(x) = x^\alpha (1-x)^{\beta_j}, \quad j = 1, \dots, r,
    \end{equation}
    where $\alpha > -1$ and $\beta_j > -1$, $j = 1, \dots, r$. 
    
    Provided $\beta_j - \beta_k \notin \mathbb{Z}$ for $j \neq k$, the weights form an AT-system on $(0,1)$. The monic Type~II polynomial $P_{\vec{n}}^{(\alpha,\vec{\beta})}(x)$ has degree $|\vec{n}|$ and satisfies:
    \begin{equation}
    \label{eq:jp_orth_beta_vec}
    \int_{0}^1 P_{\vec{n}}^{(\alpha,\vec{\beta})}(x)\, x^k \, x^\alpha (1-x)^{\beta_j} \, \mathrm dx = 0, \quad k = 0, 1, \dots, n_j - 1, \quad j = 1, \dots, r.
    \end{equation}

    It is worth emphasizing that the polynomials in Case~2 are essentially identical to those in Case~1 under the linear change of variable $x \longmapsto 1-x$. Specifically, we have the relation $$P_{\vec{n}}^{(\alpha,\vec{\beta})}(x) = (-1)^{|\vec{n}|} P_{\vec{n}}^{(\vec{\beta},\alpha)}(1-x).$$ 
    
    \paragraph{Case 3: Exponents $x^{\alpha_j} (1-x)^{\beta_j}$.}
    More generally, consider the system of weight functions on $(0,1)$ defined by
    \begin{equation}
    \label{eq:jp_weights_both_vec}
    \omega_j(x) = x^{\alpha_j} (1-x)^{\beta_j}, \quad j = 1, \dots, r,
    \end{equation}
    with $\alpha_j > -1$ and $\beta_j > -1$ for $j = 1, \dots, r$.
    
    Under appropriate conditions on the parameter vectors $\vec{\alpha}$ and $\vec{\beta}$ ensuring perfection, which holds when either $\alpha_j - \alpha_k \notin \mathbb{Z}$ or $\beta_j - \beta_k \notin \mathbb{Z}$, the monic Type~II polynomial of degree $|\vec{n}|$, denoted by $P_{\vec{n}}^{(\vec{\alpha},\vec{\beta})}(x)$, is uniquely characterized by:
    \begin{equation}
    \label{eq:jp_orth_both_vec}
    \int_{0}^1 P_{\vec{n}}^{(\vec{\alpha},\vec{\beta})}(x)\, x^k  \, x^{\alpha_j} (1-x)^{\beta_j} \, \mathrm dx = 0, \quad k = 0, 1, \dots, n_j - 1, \quad j = 1, \dots, r.
    \end{equation}
    
    In fact, taking $\beta_1=\cdots=\beta_r=\beta$ corresponds to Case 1, while choosing $\alpha_1=\cdots=\alpha_r=\alpha$ gives Case 2.

    These polynomials, despite their similarities with Jacobi--Pi\~neiro polynomials, exhibit an important difference: both parameters of the weights \eqref{eq:jp_weights_both_vec} vary. As a consequence, the favorable properties of Jacobi--Pi\~neiro, such as the Rodrigues formula or the differential equations they satisfy \cite{ABVA03, BCVA05,NVA16} may be lost. Nevertheless, these polynomials are still computable by solving the corresponding linear system.\\
    
    We now present an example of MOPs  on an unbounded interval, which generalizes Laguerre polynomials to the multiple setting.
    
    \subsubsection{Multiple Laguerre Polynomials of the First Kind}
    \label{subsubsec:multiple_laguerre}
    
    Multiple Laguerre polynomials of the first kind extend standard Laguerre polynomials to unbounded intervals $I = (0, \infty)$ \cite{BK05, LW08}. They appear frequently in random matrix theory (specifically in models with external sources), non-intersecting Brownian motions, or queuing models \cite{BK04, DK07}.
    
    Given a multi-index $\vec{n} = (n_1, \dots, n_r) \in \mathbb{N}_0^r$ and parameter vector $\vec{\alpha} = (\alpha_1, \dots, \alpha_r)$ with $\alpha_j > -1$, the weight functions are defined by
    \begin{equation}
    \label{eq:laguerre_weights}
    \omega_j(x) = x^{\alpha_j} e^{-x}, \quad j = 1, \dots, r.
    \end{equation}
    If $\alpha_j - \alpha_k \notin \mathbb{Z}$ for all $j \neq k$, the weights $\{\omega_1, \dots, \omega_r\}$ form an AT-system on $(0, \infty)$, ensuring that all multi-indices are normal. 
    
    The monic Type~II multiple Laguerre polynomial of the first kind is denoted by $L_{\vec{n}}^{\vec{\alpha}}(x)$. It is a polynomial of degree $|\vec{n}|$ satisfying the orthogonality relations:
    \begin{equation}
    \label{eq:laguerre_orth}
    \int_{0}^\infty  L_{\vec{n}}^{\vec{\alpha}}(x)\, x^k \, x^{\alpha_j} e^{-x} \,\mathrm dx = 0, \quad k = 0, 1, \dots, n_j - 1, \quad j = 1, \dots, r.
    \end{equation}

    The explicit Rodrigues formula for these monic orthogonal polynomials is given by \eqref{eq:rodrigues-jp} with $\omega_j(x)=x^{\alpha_j}e^{-x}$ and $\phi(x)=x$ \cite{ABVA03}. Deeper information about these polynomials and further references can be found in \cite[Section 3.2]{ABVA03}, \cite[Section 23.4.1]{Ism05} or \cite[Section 3.6.1]{VA20}.\\

    Having established the main univariate multiple orthogonal families, we now turn our attention to multivariate extensions. A powerful method for constructing orthogonal polynomials in two variables from univariate building blocks is Koornwinder's classical construction.
    
    \subsection{Koornwinder's Construction for Bivariate Orthogonal Polynomials}
    \label{subsec:koornwinder_construction}

    In this section, we present a well-known approach for defining OPs on planar domains based on univariate OPs. This technique, widely referred to as Koornwinder’s construction, is built on obtaining orthogonal polynomials in two variables through separation of variables. The initial groundwork was established by Larcher \cite{Lar59}, who introduced the core mechanism of generating bivariate orthogonal systems by combining one-dimensional orthogonal polynomials, an approach subsequently applied to specific classes of weight functions by Agahanov \cite{Aga65}. However, it was Koornwinder \cite{Koo75} who systematically developed and popularized this methodology, establishing a unified framework for constructing bivariate analogues of classical orthogonal polynomials over planar domains such as the unit disk, the simplex, and parabolic regions. A deeper treatment of this construction and its geometric scope is documented by Dunkl and Xu \cite[Section~2.6.1]{DX14}. Beyond its direct geometric applications, Koornwinder's construction also provides a practical framework for studying structural properties of bivariate orthogonal polynomials, such as their recurrence relations \cite{Kow82a, Kow82b}. The method is explained formally below. 

    Let $\omega_1$ and $\omega_2$ be univariate weight functions defined on the intervals $(a,b)$ and $(c,d)$, respectively. Let $\rho(x)$ be a positive function defined on $(a,b)$ satisfying one of the following two structural assumptions:
    \begin{itemize}
        \item \textbf{Case I:} $\rho(x)$ is a polynomial of degree $1$.
        \item \textbf{Case II:} $\rho(x) = \sqrt{u(x)}$, where $u(x)$ is a non-negative polynomial of degree at most $2$. In this case, we assume $c = -d < 0$ and that $\omega_2$ is an even function on $(-d, d)$.
    \end{itemize}
    
    For each $k \in \mathbb{N}_0$, let $\{p_{n}^{(k)}\}_{n\in\NN_0}$ denote the system of univariate orthonormal polynomials associated with the parameter-dependent modified weight function $\rho^{2k+1}(x) \omega_1(x)$ on $(a,b)$. In addition, let $\{q_n\}_{n\in\NN_0}$ denote the system of univariate orthonormal polynomials with respect to $\omega_2(x)$ on $(c,d)$.
    
    \begin{proposition}[\cite{DX14}, Proposition~2.6.1]
    \label{prop:koornwinder}
    For each $n \in \mathbb{N}_0$ and $0 \le k \le n$, define the bivariate polynomials $P_{n,k}(x,y)$ by
    \begin{equation}
    \label{eq:koornwinder_poly}
    P_{n,k}(x,y) \coloneqq p_{n-k}^{(k)}(x) \, [\rho(x)]^k \, q_k\left( \frac{y}{\rho(x)} \right).
    \end{equation}
    Then $\{P_{n,k} : 0 \le k \le n\}$ forms a mutually orthonormal basis for the space of bivariate polynomials of total degree $n$ with respect to the weight function
    \begin{equation}
    \label{eq:koornwinder_weight}
    W(x,y) = \omega_1(x) \, \omega_2\left( \frac{y}{\rho(x)} \right), \quad (x,y) \in R,
    \end{equation}
    where the planar domain $R \subset \mathbb{R}^2$ is given by
    \begin{equation}
    \label{eq:koornwinder_domain}
    R = \{ (x,y) \in \mathbb{R}^2 : a < x < b, \quad c \rho(x) < y < d \rho(x) \}.
    \end{equation}
    \end{proposition}
    
    This construction yields explicit orthogonal bases over classical two-dimensional regions, including triangles (simplices), disks, cones, and parabolic regions \cite{ABFMX23, DX14, FPP12, Koo75, Xu20}.\\

    Building upon these foundational concepts, the next section introduces a novel extension of Koornwinder's construction to the framework of multiple orthogonality. By synthesizing univariate multiple orthogonal polynomial systems with geometric scaling mechanisms, we establish explicit bivariate multiple orthogonal bases over non-trivial planar domains.
    
    \section{The Koornwinder construction for Multiple Orthogonal Polynomials}\label{sec:koornwinder}

    \subsection{Case 1: Combining a family of MOPs and a standard family of OP}\label{ssec:koornwinder-case-1}

    Consider a perfect system of weights $\omega_1,\dots,\omega_r$, all of them supported on the interval $[a,b]$, and a weight $\sigma$ supported on $[c,d]$. Let $\rho(x)$ be a positive function defined on $[a,b]$ such that either
    \begin{enumerate}
        \item $\rho(x)$ is a polynomial of degree 1.
        \item $\rho(x)=\sqrt{u(x)}$ with $u(x)$ being a polynomial of degree at most 2. In this case we assume that $\sigma$ is even and $[c,d]=[-d,d]$ is a symmetric interval.
    \end{enumerate}
    
    Let $\vec n =(n_1,\dots,n_r)$, $\vec k =(k_1,\dots,k_r)$ be two multi-indices such that $0\leq k_j\leq n_j$ and assume that the system of weights $\{\rho(x)^{2|\vec k|+1}\omega_j(x): 1\leq j\leq r\}$ is perfect and denote by $\{p_{\vec n}^{(\vec k)}\}_{\vec n\in\NN_0^r}$ the corresponding Type~II MOP. On the other hand, assume that the orthogonal polynomial sequence (OPS) with respect to the weight $\sigma(x)$ does exist and denote it by $\{q_n\}_{n\in\NN_0}$. We define the following functions
    \begin{equation}
        \label{eq:koornwinder-multiple-1}
        \mathbf P_{\vec n, \vec k}(x,y)=p_{\vec n-\vec k}^{(\vec k)}(x)\, [\rho(x)]^{|\vec k|} \, q_{|\vec k|}\left(\frac{y}{\rho(x)}\right).
    \end{equation}

    Observe that if  $\rho$ is either a polynomial of degree $1$ or the square root of a degree at most $2$ polynomial, then $\mathbf P_{\vec n,\vec k}$ is a polynomial of degree $|\vec n|$. Now, consider the weights
    \begin{equation}
        \label{eq:weights-1}
        W_j(x,y)=\omega_j(x)\,\sigma\left(\frac y {\rho(x)}\right),\qquad j=1,\dots,r,
    \end{equation}
    all of them supported on
    \begin{equation}
        \label{eq:support-1}
        R= \{(x,y)\in\RR^2: a<x<b, c\,\rho(x)<y<d\,\rho(x)\}.
    \end{equation}
    
    Then, in the next result we show the multiple orthogonality relations satisfied by these polynomials.

    \begin{proposition}\label{prop:orthogonality-1}
        Let $\vec n =(n_1,\dots,n_r)$, $\vec k =(k_1,\dots,k_r)$ be such that $0\leq k_j\leq n_j$ and consider the polynomial $\mathbf P_{\vec n, \vec k}$ introduced in \eqref{eq:koornwinder-multiple-1}. Then,
        \begin{equation}
            \label{eq:orthogonality-1}
            \int_R \mathbf P_{\vec n, \vec k}(x,y)\, x^t y^s\, W_j(x,y)\, \mathrm dx\,\mathrm dy =0\quad \text{ if }t+s<n_j+|\vec k|-k_j\text{ or } s<|\vec k|,\quad j=1,\dots,r. 
        \end{equation}
        In particular, $\mathbf P_{\vec n, \vec k}$ is orthogonal to polynomials of degree lower than $n_j$ with respect to $W_j$.
    \end{proposition}
    \begin{proof}
    We have that
    $$
    \int_R \mathbf P_{\vec n, \vec k}(x,y)\, x^t y^s\, W_j(x,y)\, \mathrm dx\,\mathrm dy =\int_a^b\int_{c\,\rho(x)}^{d\,\rho(x)}p_{\vec n-\vec k}^{(\vec k)}(x)\, [\rho(x)]^{|\vec k|} \, q_{|\vec k|}\left(\frac{y}{\rho(x)}\right)\, x^t y^s\, \omega_j(x)\,\sigma\left(\frac y {\rho(x)}\right)\, \mathrm dx\mathrm d y.
    $$
    
    Using the change of variables $y = \rho(x) z$, $\mathrm d y=\rho(x)\mathrm d z$, we rewrite the integral as
    $$
    \int_a^b\int_{c}^d p_{\vec n-\vec k}^{(\vec k)}(x)\, [\rho(x)]^{|\vec k|}q_{|\vec k|}(z) x^t z^s [\rho(x)]^s \omega_j(x)\,\sigma(z)\rho(x)\, \mathrm dx\mathrm d z,
    $$
    so that we can split it as the product of two univariate integrals $I_x\times I_z$, where
    \begin{equation}\label{eq:integrals-koornwinder-case-1}
        \begin{aligned}
            I_x &= \int_a^b p_{\vec n-\vec k}^{(\vec k)}(x)\,x^t\, [\rho(x)]^{|\vec k|+s+1}\omega_j(x)\,\mathrm d x,\\
            I_z &= \int_{c}^d q_{|\vec k|}(z)\, z^s\, \sigma(z)\,\mathrm d z.
        \end{aligned}
    \end{equation}
    Observe that $I_z=0$ whenever $s<|\vec k|$ due to the orthogonality of the polynomials $\{q_n\}_{n\geq 0}$. If $s\geq |\vec k|$, or equivalently $s=|\vec k|+\delta$ with $\delta$ being a non-negative integer, then $I_x$ reads
    $$
    I_x = \int_a^b p_{\vec n-\vec k}^{(\vec k)}(x)\,x^t\,\rho(x)^\delta [\rho(x)^{2|\vec k|+1}\omega_j(x)]\,\mathrm d x.
    $$
    Let us discuss the value of $I_x$ regarding $\rho(x)$.
    \begin{itemize}
        \item If $\rho(x)$ is a polynomial of degree $1$, then, by the Type~II orthogonality of the polynomials $\{p_{\vec n}^{(\vec k)}\}_{\vec n\in\NN_0^r}$ with respect to the weights $\rho(x)^{2|\vec k|+1}\omega_j(x)$, we have that $I_x=0$ whenever $t+\delta<n_j-k_j$, which using $\delta = s-|\vec k|$ yields $t+s<n_j+|\vec k|-k_j$.

        \item If $\rho(x)=\sqrt{u(x)}$ with $u(x)$ a polynomial of degree at most $2$, recall that, in this case, $\sigma$ is an even function supported on a symmetric interval. Consequently, $I_z$ in \eqref{eq:integrals-koornwinder-case-1} vanishes also whenever $\delta$ is odd. If $\delta$ is even, then focus on $I_x$:
        $$
        I_x = \int_a^b p_{\vec n-\vec k}^{(\vec k)}(x)\,x^t\,u(x)^{\delta/2} [\rho(x)^{2|\vec k|+1}\omega_j(x)]\,\mathrm d x.
        $$
        As we are assuming that $\delta$ is even, then $u(x)^{\delta/2}$ is a polynomial of degree at most $\delta$, and using Type~II orthogonality, we deduce again that $I_x=0$ if $t+s< n_j+|\vec k|-k_j$. 
    \end{itemize}    
    \end{proof}

    Besides, due to the combination of a multiple and a standard family of OPs, in this first case we are able to define Type~I MOPs  via the Koornwinder construction. The idea is as simple as substituting the Type~II polynomial $p_{\vec n-\vec k}^{(\vec k)}(x)$ in \eqref{eq:koornwinder-multiple-1} by the Type~I polynomials.

    \subsubsection{Type I Multiple Orthogonal Polynomials via the Koornwinder's construction}\label{sssec:type1Koornwinder}

    Recall that the system $\{\rho(x)^{2|\vec k|+1}\omega_j(x): 1\leq j\leq r\}$ is assumed to be perfect, so let us denote by $\{a_{\vec n,j}^{(\vec k)},\, j=1,\dots,r\}_{\vec n\in\NN^r_0}$ the associated Type~I MOP. Then, assuming $|\vec n|\ne|\vec k|$, we define
    \begin{equation}
        \label{eq:koornwinder-multiple-1-type-i}
        \mathbf A_{\vec n, \vec k,(j)}(x,y)=a_{\vec n-\vec k,j}^{(\vec k)}(x)\, [\rho(x)]^{|\vec k|} \, q_{|\vec k|}\left(\frac{y}{\rho(x)}\right),\qquad j=1,\dots,r.
    \end{equation}
    Observe that, in either case for $\rho(x)$, and since $\deg(a_{\vec n-\vec k,j}^{(\vec k)})\leq n_j-k_j-1$, $\mathbf A_{\vec n, \vec k,(j)}$ is a polynomial of degree at most $n_j+|\vec k|-k_j-1$. Next, we show that these polynomials satisfy orthogonality relations similar to the univariate Type~I MOP.

    \begin{proposition}\label{prop:orthogonality-1-type-i}
        Let $\vec n =(n_1,\dots,n_r)$, $\vec k =(k_1,\dots,k_r)$ be such that $0\leq k_j\leq n_j$, $|\vec n|\ne|\vec k|$ and consider the polynomials $\mathbf A_{\vec n, \vec k,(1)},\dots,\mathbf A_{\vec n, \vec k,(r)}$ introduced in \eqref{eq:koornwinder-multiple-1-type-i}. Then,
        \begin{equation}\label{eq:orthogonality-koornwinder-1-type-i}
            \sum_{j=1}^r\int_R \mathbf A_{\vec n, \vec k,(j)}(x,y) x^t y^s W_j(x,y)\, \mathrm dx\,\mathrm dy = 0 \qquad\text{ if } t+s<|\vec n|-1\text{ or }s<|\vec k|.
        \end{equation}
    \end{proposition}
    \begin{proof} This proof is analogous to the proof of Proposition~\ref{prop:orthogonality-1}. After the change of variables $y=\rho(x)\, z$, the integral $I$ in \eqref{eq:orthogonality-koornwinder-1-type-i} can be expressed as the product of two integrals as $I=I_x\times I_z$ where 
    \begin{equation*}
        \begin{aligned}
            I_x &= \sum_{j=1}^r \int_a^b  a_{\vec n-\vec k,j}^{(\vec k)}(x)\,x^t\,\rho(x)^{|\vec k|+s+1}\, \omega_j(x)\,\mathrm d x\\
            I_z &= \int_{c}^d q_{|\vec k|}(z)\, z^s\, \sigma(z)\,\mathrm d z.
        \end{aligned}
    \end{equation*}
    Observe that $I_z$ is the same one as in the proof of Proposition~\ref{prop:orthogonality-1}, and we know that it vanishes whenever $s<|\vec k|$. If $s=|\vec k|+\delta$ with $\delta\in\NN_0$, then
    $$
    I_x = \sum_{j=1}^r \int_a^b  a_{\vec n-\vec k,j}^{(\vec k)}(x)\,x^t\,\rho(x)^{\delta}\,\rho(x)^{2|\vec k|+1} \omega_j(x)\,\mathrm d x,
    $$
    thus
    \begin{itemize}
        \item If $\rho(x)$ is a polynomial of degree $1$, then $x^t\,\rho(x)^{\delta}$ is a polynomial of degree at most $t+\delta$ and, by Type~I orthogonality, $I_x=0$ if $t+\delta< |\vec n-\vec k|-1=|\vec n|-|\vec k|-1$. Using $\delta =s-|\vec k|$, $I_x=0$ if $t+s<|\vec n|-1$.
        
        \item If $\rho(x)=\sqrt{u(x)}$ with $u(x)$ a polynomial of degree $\leq 2$, then we know that, in this case $I_z=0$ if $\delta$ is odd. If $\delta$ is even, then $x^t\,\rho(x)^{\delta}$ is again a polynomial of degree at most $t+\delta$ and owing to the argument above, $I_x =0$ if $t+s<|\vec n|-1$ too.
    \end{itemize}
        
    \end{proof}

    Using the integrals introduced in the proof above, the following result is straightforward:
    \begin{corollary}
        Within the assumptions of Proposition~\ref{prop:orthogonality-1-type-i}, whenever $t+s=|\vec n|-1$, we have that 
        \begin{equation}\label{eq:orthogonality-koornwinder-1-type-i-2}
            \sum_{j=1}^r\int_R \mathbf A_{\vec n, \vec k,(j)}(x,y) x^t y^s W_j(x,y)\, \mathrm dx\,\mathrm dy = K \int_c^d q_{|\vec k|}(z)\, z^s \,\sigma(z)\, \mathrm dz,
        \end{equation}
        where $K$ is the leading coefficient of $\rho(x)^{s-|\vec k|}$. Indeed, the right-hand side integral vanishes if $s<|\vec k|$ and it is non-zero if $s=|\vec k|$.
    \end{corollary}
    \begin{proof}
        Just employ the normalization condition of Type~I MOPs  \eqref{eq:typeI_normalization} and the fact that $I = I_x\times I_z$.
    \end{proof} 

    Define the \textbf{Type I function} as 
    \begin{equation}
        \label{eq:type-i-function}
        \mathbf Q_{\vec n,\vec k}(x,y)=\sum_{j=1}^r \mathbf A_{\vec n,\vec k, (j)}(x,y) \, W_j(x,y),
    \end{equation}
    so that \eqref{eq:orthogonality-koornwinder-1-type-i} reads
    $$
    \int_R \mathbf Q_{\vec n,\vec k}(x,y)\, x^t y^s \,\mathrm dx\,\mathrm dy = 0 \qquad \text{ if } t+s<|\vec n|-1\text{ or }s<|\vec k|.
    $$
    With this function, which summarizes the information of the Type~I MOPs  \eqref{eq:koornwinder-multiple-1-type-i}, we present the extension of a popular result in univariate multiple orthogonality: the biorthogonality relation between Type~I and Type~II MOP \cite[Theorem~23.1.6]{Ism05}, which has been introduced in Section~\ref{subsec:mops_general}.

    \begin{proposition}\label{prop:biorthogonality-koornwinder}
        Let $\vec n,\vec k,\vec m,\vec l\in\NN^r_0$ such that $0\leq k_j\leq n_j$ and $0\leq l_j\leq m_j$, $|\vec m|\ne|\vec l|$, and consider the Type II MOP $\mathbf P_{\vec n,\vec k}$ \eqref{eq:koornwinder-multiple-1} and the Type~I function $\mathbf Q_{\vec m,\vec l}(x,y)$ \eqref{eq:type-i-function}. Then:
        $$
        \int_R \mathbf P_{\vec n,\vec k}(x,y)\, \mathbf Q_{\vec m,\vec l}(x,y)\, \mathrm dx\,\mathrm dy = 0 \qquad\text{ if }|\vec k|\neq |\vec l|.
        $$
        Moreover, if $|\vec k|=|\vec l|$:
        \begin{equation}\label{eq:biorthogonality-koornwinder}
            \int_R \mathbf P_{\vec n,\vec k}(x,y)\, \mathbf Q_{\vec m,\vec l}(x,y)\, \mathrm dx\,\mathrm dy = \begin{cases}
                
                0 & \text{ if }\vec m-\vec l \leq \vec  n-\vec k \quad \text{(componentwise)},\\
                0& \text{ if }|\vec n|\leq |\vec m|-2,\\
                \|q_{|\vec k|}\|^2&  \text{ if }|\vec n|= |\vec m|-1,\\
            \end{cases}
        \end{equation}
        where $\|q_{|\vec k|}\|^2=\displaystyle\int_c^d (q_{|\vec k|}(x))^2\sigma(x)\,\mathrm dx$ is the squared norm of the polynomial $q_{|\vec k|}$.
    
    \end{proposition}
    \begin{proof}
        Using \eqref{eq:koornwinder-multiple-1}, \eqref{eq:koornwinder-multiple-1-type-i}, \eqref{eq:type-i-function}, \eqref{eq:weights-1} and the change of variables $y = \rho(x)\, z$, the integral might be expressed as $I=I_x\times I_z$ where
        \begin{equation*}
            \begin{aligned}
                I_x  &= \sum_{j=1}^r \int_a^b p_{\vec n-\vec k}^{(\vec k)}(x)\,a_{\vec m-\vec l,j}^{(\vec l)}(x)\,\rho(x)^{|\vec k|+|\vec l|+1} \,\omega_j(x)\,\mathrm dx, \\
                I_z &= \int_c^d q_{|\vec k|}(z)\,q_{|\vec l|}(z)\,\sigma(z)\, \mathrm dz.
            \end{aligned}
        \end{equation*}
        It is clear that $I_z =0$ if $|\vec k|\neq|\vec l|$ and $I_z = \|q_{|\vec k|}\|^2=\displaystyle\int_c^d (q_{|\vec k|}(z))^2\sigma(z)\,\mathrm dz$ if $|\vec k|=|\vec l|$.

        Assume that $|\vec k|=|\vec l|$. Then 
        $$
        I_x  = \sum_{j=1}^r \int_a^b p_{\vec n-\vec k}^{(\vec k)}(x)\,a_{\vec m-\vec l,j}^{(\vec l)}(x)\,\rho(x)^{2|\vec k|+1} \,\omega_j(x)\,\mathrm dx,
        $$
        and the result follows by using the univariate biorthogonality \eqref{eq:biorthogonality}.
    \end{proof}

    This result becomes more compelling when the multi-indices involved lie along a path of neighbour multi-indices; see the result below.

    \begin{corollary}
        Let $\Omega=\{\vec m_i:\,i\geq 0\}$ be a path of neighbour multi-indices. That is, $|\vec m_i|=i$, $i\in\NN_0$ and $\vec m_{i+1}=\vec m_i+\vec e_j$, where $\vec e_j =(0,\dots,0,\overset{(j)}{1},0,\dots,0)$. Choose $\vec n,\vec k,\vec m,\vec l\in \Omega$ such that $|\vec k|\leq|\vec n|$, $|\vec l|\le |\vec m|$ and denote $n=|\vec n|,k=|\vec k|,m=|\vec m|, l=|\vec l|$. Then,
        $$
        \int_R \mathbf P_{\vec n,\vec k}(x,y)\, \mathbf Q_{\vec m,\vec l}(x,y)\, \mathrm dx\,\mathrm dy =\|q_{k}\|^2\, \delta_{k,l}\,\delta_{n,m-1},
        $$
        where $\delta_{i,j}$ denotes the Kronecker delta $\delta_{i,j}=1$ if $i=j$ and $\delta_{i,j}=0$ otherwise. 
    \end{corollary}
    \begin{proof}
        The proof is straightforward using Proposition~\ref{prop:biorthogonality-koornwinder} and the fact that, in $\Omega$, $\vec m_{i}\leq \vec m_{i'}$ if and only if $i\leq i'$ for every $\vec m_i,\vec m_{i'}\in \Omega$.
    \end{proof}

    As presented, the proposed construction exhibits several features inherited from univariate multiple orthogonality. Moreover, combining a multiple family with a standard family of OPs enables us to define Type~I and Type~II bivariate MOP. In the following section, we merge two families of MOPs.

    \subsection{Case 2: Combining two families of MOP}\label{ssec:koornwinder-case-2}
    Consider two perfect systems of $r$ weights: $\omega_1,\dots,\omega_r$ supported on $[a,b]$ and $\sigma_1,\dots,\sigma_r$ supported on $[c,d]$. Let $\rho(x)$ be a polynomial of degree at most $1$ which is positive on $[a,b]$. Observe that the case of $\rho(x)$ being the square root of a polynomial of degree at most $2$ is no longer considered. The reason for this modification will be explained shortly, in Remark~\ref{rem:rho-sqrt-pol}.

    Let $\vec n=(n_1,\dots,n_r),\vec k=(k_1,\dots,k_r)\in\NN_0^r$ be two multi-indices with $0\leq k_j\leq n_j$. Assume that the system $\{\rho(x)^{|\vec k|+k_j+1}\omega_j(x):1\leq j\leq r\}$ is perfect and denote by $\{p_{\vec n}^{(\vec k)}\}_{\vec n\in\NN^r_0}$ the corresponding Type~II MOP. Denote by $\{q_{\vec n}\}_{\vec n\in\NN^r_0}$ the Type~II MOP with respect to the system $\sigma_1,\dots,\sigma_r$. Then, define the polynomials
    \begin{equation}
        \label{eq:koornwinder-multiple-2}
        \mathbf P_{\vec n, \vec k}(x,y)=p_{\vec n-\vec k}^{(\vec k)}(x)\, [\rho(x)]^{|\vec k|} \, q_{\vec k}\left(\frac{y}{\rho(x)}\right).
    \end{equation}
    \begin{remark}\label{rem:rho-sqrt-pol}
        If $\rho(x)=\sqrt{u(x)}$ with $\deg(u(x))\leq 2$, then
        $$
        [\rho(x)]^{|\vec k|} q_{\vec k}\left(\frac{y}{\rho(x)}\right) = c_0 u(x)^{|\vec k|/2} + c_1 u(x)^{(|\vec k|-1)/2}y+\cdots+ c_{|\vec k|}y^{|\vec k|}
        $$
        so that there might be terms which are not polynomials.
        \end{remark}
    
    Consider the bivariate weights
    \begin{equation}
        \label{eq:weights-2}
        W_j(x,y)=\omega_j(x)\,\sigma_j\left(\frac y {\rho(x)}\right),
    \end{equation}
    supported on $R$ given in \eqref{eq:support-1}.
    
    With these elements, we are about to prove the multiple orthogonality relations satisfied by \eqref{eq:koornwinder-multiple-2}.
    \begin{proposition}\label{prop:orthogonality-2}
        Let $\vec n=(n_1,\dots,n_r),\vec k=(k_1,\dots,k_r)\in\NN_0^r$ be two multi-indices with $0\leq k_j\leq n_j$ and consider the polynomial $\mathbf P_{\vec n, \vec k}(x,y)$ defined in \eqref{eq:koornwinder-multiple-2}. Then, the following relations hold:
        \begin{equation}
            \label{eq:orthogonality-2}
            \int_R \mathbf P_{\vec n, \vec k}(x,y)\, x^t y^s\, W_j(x,y)\, \mathrm dx\,\mathrm dy =0\qquad\text{ if }t+s<n_j\text{ or }s<k_j,\quad j=1,\dots,r.
        \end{equation}
        
    \end{proposition}
    \begin{proof}
        This proof is analogous to that of Proposition~\ref{prop:orthogonality-1}. Using  the change of variables $y=\rho(x)z$ it is possible to write the integral as the product of two integrals $I_x\times I_z$ where
        \begin{equation*}
        \begin{aligned}
            I_x &= \int_a^b p_{\vec n-\vec k}^{(\vec k)}(x)\,x^t\, [\rho(x)]^{|\vec k|+s+1}\omega_j(x)\,\mathrm d x,\\
            I_z &= \int_{c}^d q_{\vec k}(z)\, z^s\, \sigma_j(z)\,\mathrm d z.
        \end{aligned}
    \end{equation*}
    By Type~II orthogonality of the polynomials $ q_{\vec k}$, we have that $I_z=0$ if $s<k_j$.

    Assume that $s=k_j+\delta$, with $\delta$ a non-negative integer. Then, $I_x$ might be rewritten as
    $$
    \int_a^b p_{\vec n-\vec k}^{(\vec k)}(x)\,x^t\rho(x)^\delta\, [\rho(x)^{|\vec k|+k_j+1}\omega_j(x)]\,\mathrm d x.
    $$
    As $\rho(x)$ is a polynomial of degree $1$, $I_x=0$ whenever $t+\delta<n_j-k_j$, and using $\delta=s-k_j$, $I_x=0$ whenever $t+s<n_j$. 

    \end{proof}

    \begin{remark}
        Concerning the lack of a definition of Type~I MOPs  using \eqref{eq:koornwinder-multiple-2}, notice that in the construction of  Section~\ref{sssec:type1Koornwinder}, in the proof of Proposition~\ref{prop:orthogonality-1-type-i}, after the change of variables, the resulting integral can be written as $I_z\sum_{j=1}^r I_x(j) $ where the factor $I_z$ is independent of $j$. This separation allows one to factor out $I_z$ and then apply the Type I orthogonality relations to the remaining sum over $j$. However, when two families of MOPs  are combined, the corresponding integral takes the form $\sum_{j=1}^r I_x(j)\times I_z(j)$ and both factors depend on $j$. Consequently, the argument used to construct Type I bivariate MOPs  does not extend in a straightforward way, and only a Type II construction is obtained in this setting.
    \end{remark}

    Now, we apply this method to construct MOPs  on several planar domains. 
    
    \section{Multiple Orthogonal Polynomials on planar domains}\label{sec:examples}

    In this section, we employ the Koornwinder construction described above to define MOPs  on different domains. Specifically, we present MOPs  on a parabolic domain (Section~\ref{ssec:MOP-parabolic}), on a square (Section~\ref{ssec:MOP-square}), on the first quadrant (Section~\ref{ssec:MOP-quadrant}), on a wedge  (Section~\ref{ssec:MOP-wedge}) and on the triangle (Section~\ref{ssec:MOP-triangle}), see Figure~\ref{fig:domains}.

    \begin{figure}[ht]
        \centering
        \begin{tabular}{ccc}
             \includegraphics[height=3.5cm]{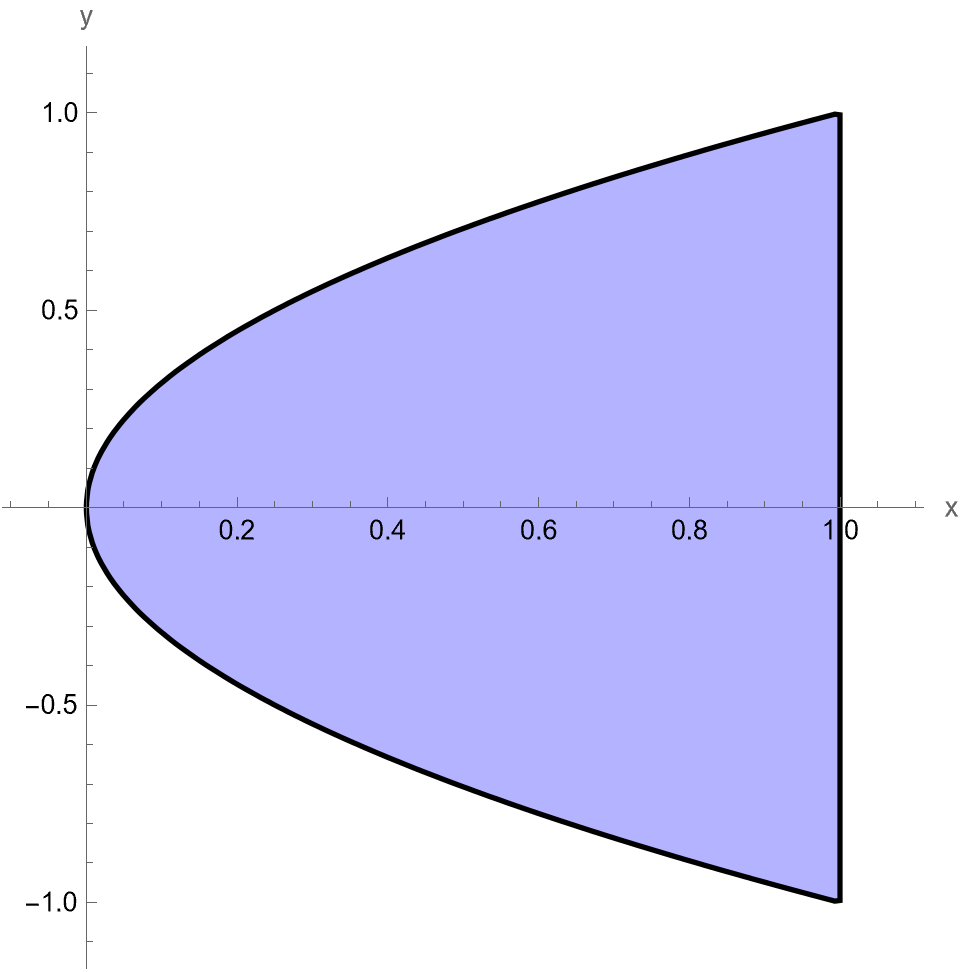} & 
             \includegraphics[height=3.5cm]{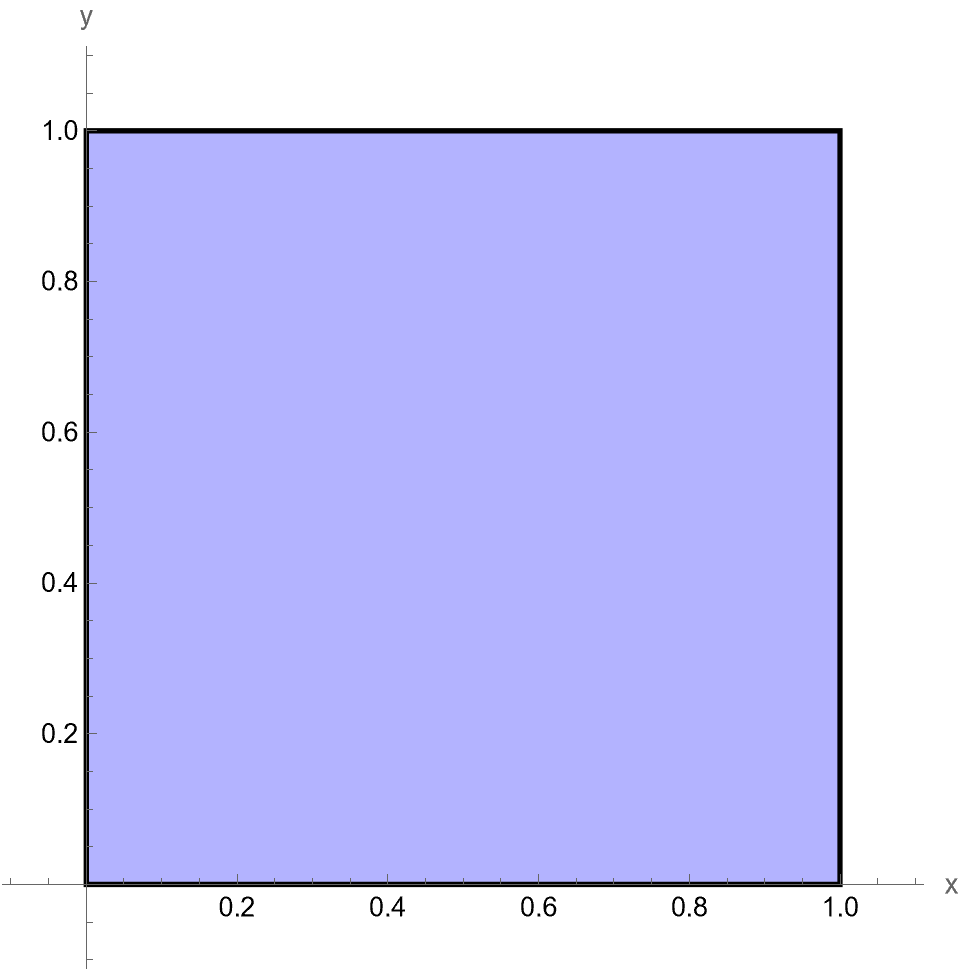} & 
            \includegraphics[height=3.5cm]{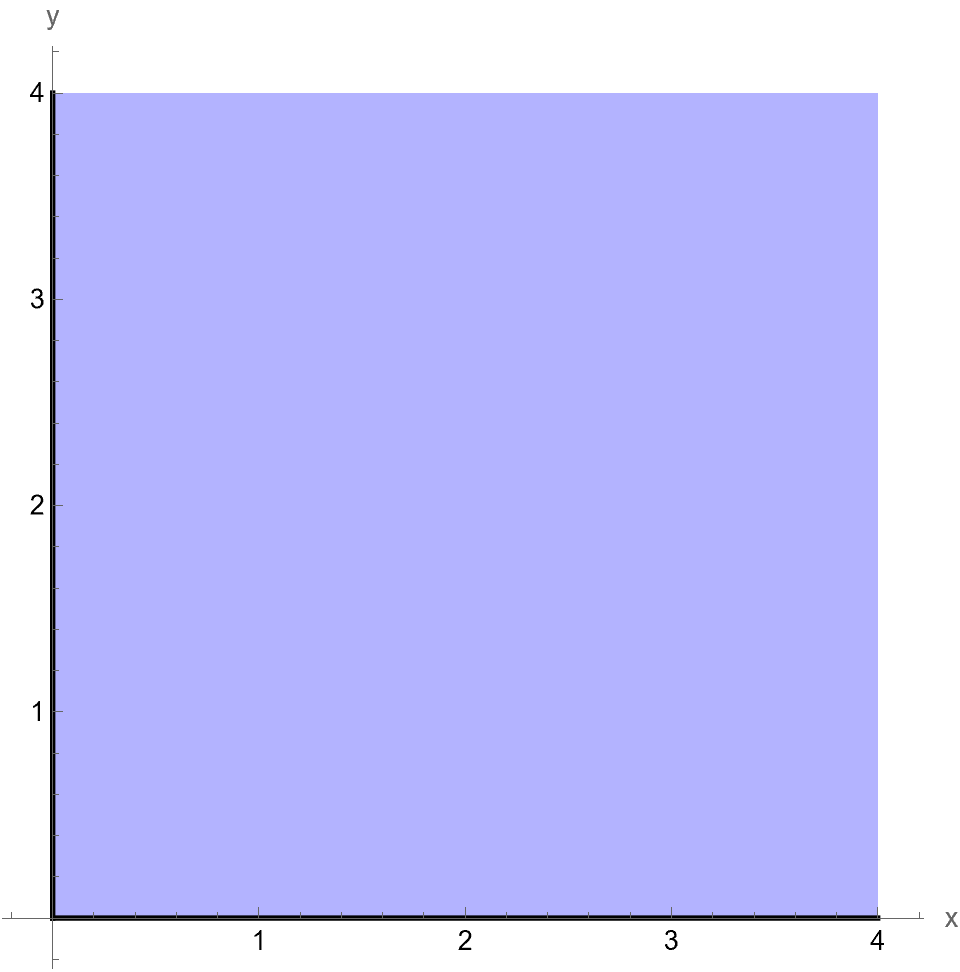} \\
             (a) Parabolic domain & (b) Square $[0,1]^2$ & (c)  First quadrant $(\RR_0^+)^2$ 
        \end{tabular}
        \begin{tabular}{cc}
             \includegraphics[height=3.5cm]{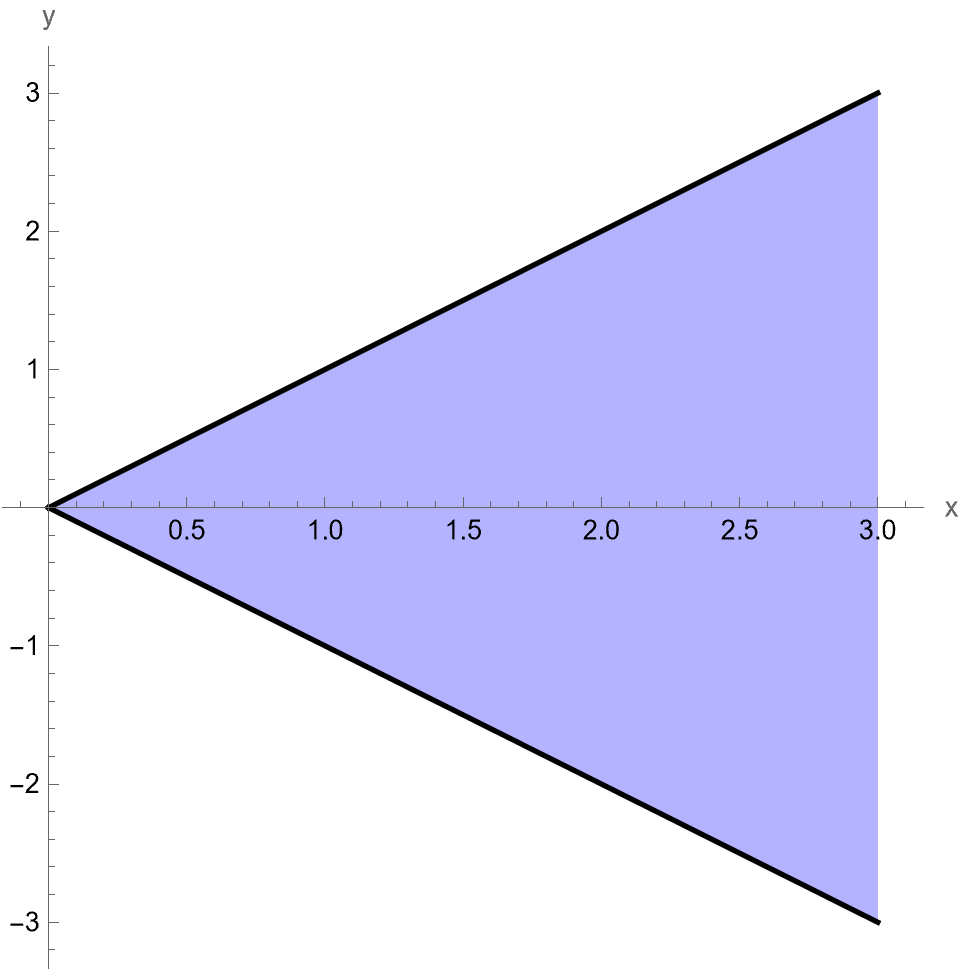} & 
             \includegraphics[height=3cm]{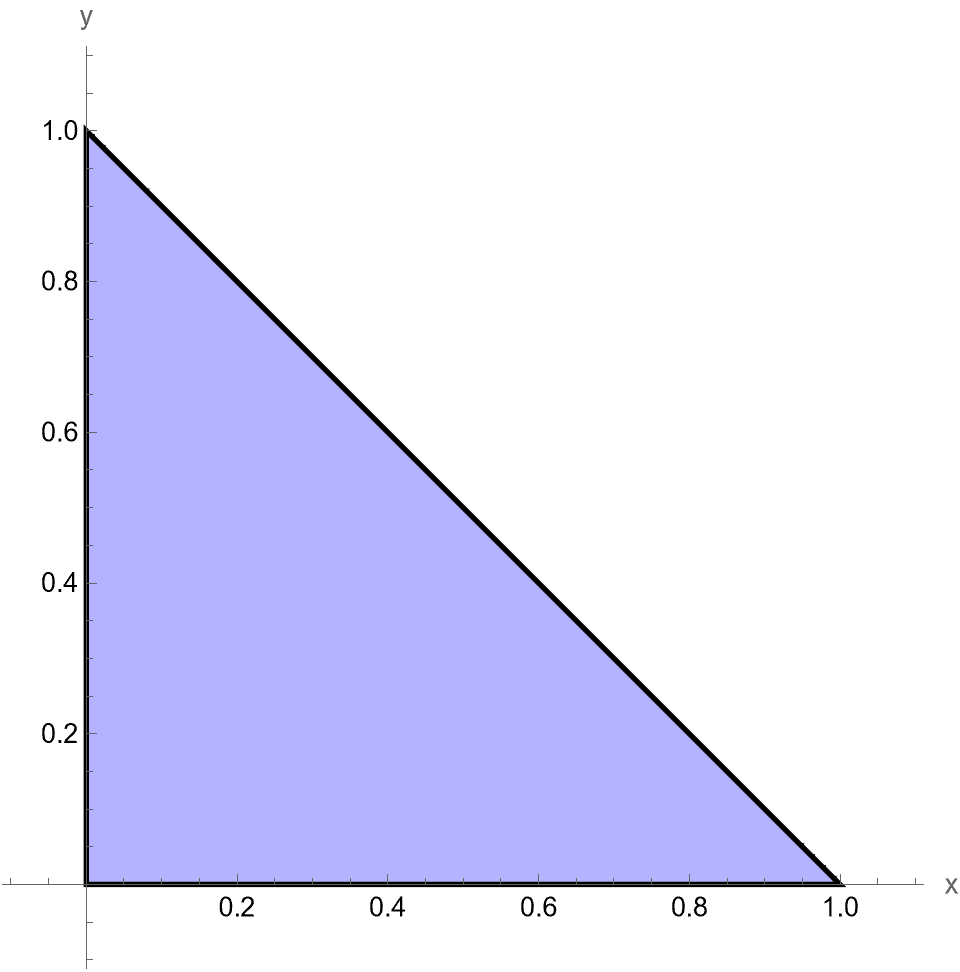}  \\
             (d) Wedge $\mathbb V^2$ & (e) Triangle $T$ 
        \end{tabular}
        \caption{The two-dimensional domains considered in this paper}
        \label{fig:domains}
    \end{figure}

    \subsection{Multiple Orthogonal Polynomials on a Parabolic Domain}\label{ssec:MOP-parabolic}

    First, we employ the Koornwinder construction to introduce MOPs  on a planar domain bounded by an algebraic curve: the parabolic region $R = \left\{ (x,y) \in \mathbb{R}^2 : y^2 < x < 1 \right\}$. Orthogonal polynomials on $R$ associated with weight functions of the form $W_{\alpha,\beta}(x,y) = (x-y^2)^\alpha (1-x)^\beta $, have been described in \cite{Koo75, DX14}. To the best of our knowledge, the results developed in this section represent the first extension of multiple orthogonal polynomials to a parabolic domain to date.

    We will use the method explained in Section~\ref{ssec:koornwinder-case-1} (Case 1), where we combine a family of MOPs  with a standard orthogonal polynomials sequence to obtain MOPs. Then, consider $\alpha,\beta_1,\dots,\beta_r>-1$ such that $\beta_j-\beta_i\notin\ZZ$ whenever $i\ne j$ and define the weights
    \begin{align}
        \omega_j(x)&=x^\alpha(1-x)^{\beta_j}, \quad x\in[0,1], \quad j=1,\dots,r, & \sigma(x)=(1-x^2)^\alpha,\quad x\in[-1,1],
    \end{align}
    and the function $\rho(x)=\sqrt{x}$. Let $\vec n=(n_1,\dots,n_r)$ and $\vec k=(k_1,\dots,k_r)$ be two multi-indices such that $0\leq k_j\leq n_j$, $j=1,\dots,r$. The MOPs  with respect to the weights $\rho(x)^{2|\vec k|+1}\omega_j(x)=x^{\alpha+|\vec k|+1/2}(1-x)^{\beta_j}$ on $[0,1]$ are the Jacobi--Pi\~neiro polynomials denoted by $P^{(\alpha+|\vec k|+1/2,\vec\beta)}_{\vec n}(x)$ (Type~II) and $A^{(\alpha+|\vec k|+1/2,\vec\beta)}_{\vec n,j}(x)$, $j=1,\dots,r$ (Type~I); which indeed are the standard Jacobi--Pi\~neiro polynomials under the substitution $x\mapsto 1-x$: $P^{(\vec\beta,\alpha+|\vec k|+1/2)}_{\vec n}(1-x)$ and $A^{(\vec\beta,\alpha+|\vec k|+1/2)}_{\vec n,j}(1-x)$, introduced in Section~\ref{subsubsec:jacobi_pineiro} (Case 2). On the other hand, the OP with respect to the weight $(1-x^2)^\alpha$ are the Gegenbauer polynomials $C^{(\alpha+1/2)}_n(x)$, explained in Section~\ref{subsubsec:gegenbauer_standard}.

    Then, we define the polynomials
    \begin{equation}\begin{aligned}
        \mathbf{P}_{\vec n,\vec k}^{(\alpha,\vec\beta)}(x,y) &= P^{(\alpha+|\vec k|+1/2,\vec\beta)}_{\vec n-\vec k}(x)\,(\sqrt{x})^{|\vec k|}\,C_{|\vec k|}^{(\alpha+1/2)}\left(\dfrac{y}{\sqrt x}\right),\\
        \mathbf{A}_{\vec n,\vec k,(j)}^{(\alpha,\vec\beta)}(x,y) &= A^{(\alpha+|\vec k|+1/2,\vec\beta)}_{\vec n-\vec k,j}(x)\,(\sqrt{x})^{|\vec k|}\,C_{|\vec k|}^{(\alpha+1/2)}\left(\dfrac{y}{\sqrt x}\right),\qquad j=1,\dots,r.
    \end{aligned}   
    \end{equation}
    Regarding \eqref{eq:weights-2} and \eqref{eq:support-1}, introduce the following weights and their corresponding support (see Figure~\ref{fig:domains}(a)).
    \begin{equation}
        \label{eq:weights-parabolic}
        W_j(x,y)=(x-y^2)^\alpha (1-x)^{\beta_j},\qquad(x,y)\in R=\{(x,y)\in\RR^2\,:\,0<x<1,\ y^2<x\}.
    \end{equation}
    By Proposition~\ref{prop:orthogonality-1}, we have that
    $$
    \int_R \mathbf{P}_{\vec n,\vec k}^{(\alpha,\vec\beta)}(x,y) \, x^t y^s\, (x-y^2)^\alpha (1-x)^{\beta_j}\, \mathrm dx\, \mathrm dy =0\qquad\text{ if } t+s<n_j+|\vec k|-k_j\text{ or } s<|\vec k|, \quad j=1,\dots,r.
    $$
    By Proposition~\ref{prop:orthogonality-1-type-i},
    $$
    \sum_{j=1}^r \int_R \mathbf{A}_{\vec n,\vec k,(j)}^{(\alpha,\vec\beta)}(x,y) \, x^t y^s\, (x-y^2)^\alpha (1-x)^{\beta_j}\, \mathrm dx\, \mathrm dy =0\qquad\text{ if } t+s<|\vec n|-1\text{ or } s<|\vec k|.
    $$\\

    In the following section, we present the first example of Case 2 from Section~\ref{sec:koornwinder}, in which two families of MOPs  are combined. We begin with a simple example that results in the tensor product of univariate MOPs : the square.

    \subsection{Multiple Orthogonal Polynomials on the square}\label{ssec:MOP-square}

    A natural question about the extension of MOPs  to the bivariate case concerns how the product of two families of univariate MOPs  works. Indeed, this question was first addressed in \cite{FV26a} for arbitrary systems of measures. In this section, it is shown that the tensor product of univariate Jacobi--Pi\~neiro MOPs  is a particular case of the Koornwinder method we introduce.

    Now, we focus on the techniques from Section~\ref{ssec:koornwinder-case-2} (Case 2). Let $\alpha_1,\dots,\alpha_r,\beta,\gamma_1,\dots,\gamma_r,\lambda>-1$ be such that $\alpha_j-\alpha_i, \gamma_j-\gamma_i\notin \ZZ$ if $i\neq j$, consider the weights
    \begin{equation}
        \omega_j(x) =x^{\alpha_j}(1-x)^{\beta},\qquad \sigma_j(x)=x^{\gamma_j}(1-x)^\lambda,\qquad x\in[0,1],\quad j=1,\dots,r,
    \end{equation}
    and choose $\rho(x)=1$. Let $\vec n=(n_1,\dots,n_r)$ and $\vec k=(k_1,\dots,k_r)$ be two multi-indices such that $0\leq k_j\leq n_j$, $j=1,\dots,r$ and denote $\vec\alpha=(\alpha_1,\dots,\alpha_r)$ and $\vec\gamma=(\gamma_1,\dots,\gamma_r)$. It is clear that the MOPs  with respect to the weights $\{\rho(x)^{|\vec k|+k_j+1}\omega_j(x):\,1\leq j\leq r\}=\{\omega_1,\dots,\omega_r\}$ and $\{\sigma_1,\dots,\sigma_r\}$ are the Jacobi--Pi\~neiro polynomials $P^{(\vec\alpha,\beta)}_{\vec n}(x)$ and $P^{(\vec\gamma,\lambda)}_{\vec n}(x)$, given in Section~\ref{subsubsec:jacobi_pineiro} (Case 1). 

    Using \eqref{eq:koornwinder-multiple-2}, \eqref{eq:weights-2} and \eqref{eq:support-1}, we define the polynomials
    \begin{equation*}
        \mathbf P^{(\vec\alpha,\beta,\vec\gamma,\lambda)}_{\vec n,\vec k}(x,y)= P^{(\vec\alpha,\beta)}_{\vec n-\vec k}(x)P^{(\vec\gamma,\lambda)}_{\vec k}(y),
    \end{equation*}
    and the weights $W_j(x,y)=x^{\alpha_j}y^{\gamma_j}(1-x)^\beta(1-y)^\lambda$, $j=1,\dots,r$,  supported on the square $[0,1]\times[0,1]$ (Figure~\ref{fig:domains}(b)).

    Then, the following multiple orthogonality conditions hold
    $$
    \int_0^1\int_0^1  \mathbf P^{(\vec\alpha,\beta,\vec\gamma,\lambda)}_{\vec n,\vec k}(x,y)\, x^t y^s \, x^{\alpha_j}y^{\gamma_j}(1-x)^\beta(1-y)^\lambda\, \mathrm dx\, \mathrm dy =0\quad\text{ if }t+s<n_j\text{ or }s<k_j,\quad j=1,\dots,r.
    $$

    Observe the analogy between this construction and the one in \cite[Section 2.6.1]{DX14}, where Jacobi polynomials are employed to define standard OPs on the square $[-1,1]\times[-1,1]$ as the product of two univariate families of OPs.

    Note that, in the considered construction, the parameters $\alpha$ and $\gamma$ are varied in order to generate MOP. More generally, any alternative choice of two free parameters can be employed to obtain MOPs  on the square with respect to the corresponding family of weight functions. Furthermore, by holding three parameters fixed and varying only one, the Case~1 construction (Section~\ref{ssec:koornwinder-case-1}) can be applied to form MOPs  on $[0,1]^2$ by coupling Jacobi--Pi\~neiro polynomials with appropriately shifted classical Jacobi polynomials.
    
    %This means, assuming that the considered parameters satisfy the AT condition for the existence of Jacobi--Pi\~neiro polynomials:
    
    %\begin{itemize}
    %    \item Define $\vec\alpha=(\alpha_1,\dots,\alpha_r),\beta,\gamma,\vec\lambda=(\lambda_1,\dots,\lambda_r)$. Then $\mathbf P^{(\vec\alpha,\beta,\gamma,\vec\lambda)}_{\vec n,\vec k}(x,y)= P^{(\vec\alpha,\beta)}_{\vec n-\vec k}(x)P^{(\gamma,\vec\lambda)}_{\vec k}(y)$ are MOP with respect to $x^{\alpha_j}y^\gamma (1-x)^{\beta}(1-y)^{\lambda_j}$, $j=1,\dots,r$.
    %    \item Define $\alpha,\vec\beta=(\beta_1,\dots,\beta_r),\vec\gamma=(\gamma_1,\dots,\gamma_r),\lambda$. Then $\mathbf P^{(\alpha,\vec\beta,\vec\gamma,\lambda)}_{\vec n,\vec k}(x,y)= P^{(\alpha,\vec\beta)}_{\vec n-\vec k}(x)P^{(\vec\gamma,\lambda)}_{\vec k}(y)$ are MOP with respect to $x^{\alpha}y^{\gamma_j} (1-x)^{\beta_j}(1-y)^{\lambda}$, $j=1,\dots,r$.
    %    \item Define $\alpha,\vec\beta=(\beta_1,\dots,\beta_r),\gamma,\vec\lambda=(\lambda_1,\dots,\lambda_r)$. Then $\mathbf P^{(\alpha,\vec\beta,\gamma,\vec\lambda)}_{\vec n,\vec k}(x,y)= P^{(\alpha,\vec\beta)}_{\vec n-\vec k}(x)P^{(\gamma,\vec\lambda)}_{\vec k}(y)$ are MOP with respect to $x^{\alpha}y^{\gamma} (1-x)^{\beta_j}(1-y)^{\lambda_j}$, $j=1,\dots,r$.
    %\end{itemize}

    Next, we present the first example of MOPs  on an unbounded domain based on multiple Laguerre polynomials.

    \subsection{Multiple Orthogonal Polynomials on the first quadrant}\label{ssec:MOP-quadrant}

    One of the most employed unbounded domains is the first quadrant, which can be denoted by $(\RR_0^+)^2$ or $[0,+\infty)^2$, see Figure~\ref{fig:domains}(c). In \cite[Section~5.6]{FPP12}, the Koornwinder construction is utilized to get OPs on $(\mathbb R_0^+)^2$ with respect to a semi-classical weight. Using multiple Laguerre polynomials, this construction is extended to the multiple case. Thus, we use the tools from Section~\ref{ssec:koornwinder-case-2}.

    Consider $\beta_1,\dots,\beta_r>-1$ and $\alpha_1,\dots,\alpha_r$ such that $\alpha_j+\beta_j>-1$, $j=1,\dots,r$, and such that  $\beta_j-\beta_i,\alpha_j+\beta_j-\alpha_i-\beta_i\notin\ZZ$ whenever $i\ne j$, and define the weights
    \begin{equation}
        \omega_j(x)=x^{\alpha_j+\beta_j}e^{-x}\quad x\in\RR_0^+,\qquad\sigma_j(x)=x^{\beta_j}e^{-x},\quad x\in\mathbb R^+_0,
    \end{equation}
    for every $j=1,\dots,r$. We use $\rho(x)=x$. Now, for $\vec n=(n_1,\dots,n_r)$, $\vec k=(k_1,\dots,k_r)\in\NN^r_0$ with $0\le k_j\le n_j$, $j=1,\dots,r$, we consider MOPs  with respect to $\rho(x)^{|\vec k|+k_j+1}\omega_j(x)=x^{\alpha_j+\beta_j+|\vec k|+k_j+1}\,e^{-x}$, which are the multiple Laguerre polynomials of the first kind $L_{\vec n}^{\widetilde\alpha}(x)$ with $(\widetilde\alpha)_j=\alpha_j+\beta_j+|\vec k|+k_j+1$, introduced in Section~\ref{subsubsec:multiple_laguerre}. Besides, we need MOPs  with respect to the weights $\sigma_j(x)$, which clearly are the multiple Laguerre polynomials of the first kind $L^{\vec\beta}_{\vec n}(x)$. Regarding \eqref{eq:koornwinder-multiple-2}, \eqref{eq:weights-2} and \eqref{eq:support-1}, define the polynomials
    $$
    \mathbf L_{\vec n,\vec k}^{\vec\alpha,\vec\beta}(x,y)=L_{\vec n-\vec k}^{\widetilde\alpha}(x)\, x^{|\vec k|}\, L_{\vec k}^{\vec \beta}\left(\dfrac y x \right),
    $$
    and the semiclassical weight
    $$
    W_j(x,y)=x^{\alpha_j}\,y^{\beta_j}\, e^{-(x+\frac y x)}, \qquad (x,y)\in(\mathbb R^+_0)^2.
    $$
    Then, we have that these polynomials satisfy
    $$
    \int_0^\infty\int_0^\infty \mathbf L_{\vec n,\vec k}^{\vec\alpha,\vec\beta}(x,y)\, x^t y^s \, x^{\alpha_j}\,y^{\beta_j}\, e^{-(x+\frac y x)}\, \mathrm dx\,\mathrm dy  =0\qquad\text{ if }t+s<n_j\text{ or }s<k_j,\qquad j=1,\dots,r.
    $$

    Now, we apply Case 1 of the Koornwinder construction in Section~\ref{sec:koornwinder} to get MOPs  on a different unbounded domain that was also considered in the standard orthogonal polynomial setting in \cite{FPP12}.

    \subsection{Multiple Orthogonal Polynomials on the Wedge} \label{ssec:MOP-wedge}

    As another example of unbounded support, we will introduce MOPs  on the domain known as the ``wedge'', also known as the two-dimensional ``cone'', see Figure~\ref{fig:domains}(d). This construction was first introduced in \cite{FPP12}, and further studied and extended in \cite{ABFMX23, Xu20}; these polynomials are also called bivariate Laguerre--Jacobi polynomials, as they combine Laguerre and Jacobi polynomials. In our multiple orthogonal extension, we will use Laguerre MOPs  together with Jacobi (Gegenbauer) standard OPs. We therefore apply the tools from Section~\ref{ssec:koornwinder-case-1} (Case 1).
    Let $\mu>-1$ and $\alpha_1,\dots,\alpha_r$ be such that $\alpha_j+2\mu>-1$ and such that $\alpha_j-\alpha_i\notin\ZZ$ if $i\neq j$. Introduce the weights
    \begin{equation}
        \omega_j(x)=x^{\alpha_j+2\mu}e^{-x},\quad x\in\RR^+,\quad j=1,\dots,r, \qquad \sigma(x)=(1-x^2)^\mu,\quad x\in[-1,1],
    \end{equation}
    and the function $\rho(x)=x$. Consider two multi-indices $\vec n=(n_1,\dots,n_r)$ and $\vec k=(k_1,\dots,k_r)$ such that $0\leq k_j\leq n_j$, $j=1,\dots,r$. MOPs  with respect to the weights $\rho(x)^{2|\vec k|+1}\omega_j(x)=x^{\alpha_j+2\mu+2|\vec k|+1}e^{-x}$, $j=1,\dots,r$ are the multiple Laguerre polynomials of the first kind (introduced in Section~\ref{subsubsec:multiple_laguerre}), which will be denoted as $L_{\vec n}^{\tilde\alpha}(x)$ (Type~II) and $B_{\vec n,j}^{\tilde\alpha}(x)$ (Type~I), where the components of $\tilde\alpha\in\RR^r$ are $(\tilde\alpha)_j=\alpha_j+2|\vec k|+2\mu+1$, $j=1,\dots,r$. On the other hand, the standard OPs with respect to $\sigma(x)$ are the Gegenbauer polynomials $C_n^{\mu+1/2}(x)$, explained in Section~\ref{subsubsec:gegenbauer_standard}.

    Using \eqref{eq:koornwinder-multiple-1} and \eqref{eq:koornwinder-multiple-1-type-i} we introduce:
    \begin{equation*}
    \begin{aligned}
        \mathbf L^{(\vec\alpha,\mu)}_{\vec n,\vec k}(x,y)&=L_{\vec n-\vec k}^{(\tilde\alpha)}(x)\, x^{|\vec k|}\, C_{|\vec k|}^{(\mu+1/2)}\left(\dfrac y x\right).\\
        \mathbf B^{(\vec\alpha,\mu)}_{\vec n,\vec k,(j)}(x,y)&=B_{\vec n-\vec k,j}^{(\tilde\alpha)}(x)\, x^{|\vec k|}\, C_{|\vec k|}^{(\mu+1/2)}\left(\dfrac y x\right),\qquad j=1,\dots,r.
    \end{aligned}
    \end{equation*}
    By \eqref{eq:weights-1} and \eqref{eq:support-1}, consider the weights
    $$
    W_j(x,y)=x^{\alpha_j}e^{-x}(x^2-y^2)^\mu,\qquad (x,y)\in \mathbb V^2=\{(x,y)\in\RR^2\,:\,  x>0, \, -x<y<x\}.
    $$
    The domain $\mathbb V^2$ is called the wedge, and it is the bivariate version of the cone, understanding the cone as the set $ \mathbb V^{d+1} =\{(\mathbf x,t)\in\RR^{d+1}\, :\, \|\mathbf x\|\leq t,\, \mathbf x\in\RR^d, t\geq 0\} $ \cite{ABFMX23} (see Section~\ref{sec:MOP-cone}). By Propositions~\ref{prop:orthogonality-1}~and~\ref{prop:orthogonality-1-type-i}, we have that
    \begin{equation*}
    \begin{aligned}
        &\int_{\mathbb V^2} \mathbf L^{(\vec\alpha,\mu)}_{\vec n,\vec k}(x,y)\, x^t y^s\, x^{\alpha_j}e^{-x}(x^2-y^2)^\mu\,\mathrm dx\,\mathrm dy = 0\qquad\text{ if } t+s<n_j+|\vec k|-k_j\text{ or } s<|\vec k|, \quad j=1,\dots,r.\\
        &\sum_{j=1}^r\int_{\mathbb V^2} \mathbf B^{(\vec\alpha,\mu)}_{\vec n,\vec k,(j)}(x,y)\, x^t y^s\, x^{\alpha_j}e^{-x}(x^2-y^2)^\mu\,\mathrm dx\,\mathrm dy = 0\qquad\text{ if }t+s<|\vec n|-1\text{ or } s<|\vec k|.\\
    \end{aligned}
    \end{equation*}

    As one last example, we study MOPs  on a triangle.

    \subsection{Multiple Orthogonal Polynomials on the Triangle}\label{ssec:MOP-triangle} 
    
    Multiple Orthogonal Polynomials on the triangle were first studied in \cite{FFMV26}, where a definition in terms of a Rodrigues-type formula, extending \eqref{eq:rodrigues-jp}, was given. However, the non-uniqueness of OPs on multivariate domains is well known. Therefore, it is to be expected that MOPs  are not unique either. In this section, we will present a new definition of MOPs  on the triangle using the method introduced in Section~\ref{ssec:koornwinder-case-2}. Furthermore, we present another extension of OPs on the triangle based on a basis that is no longer constructed by Koornwinder's method, but that still admits a representation in terms of Jacobi polynomials, see \cite[Section 4.1]{Xu17}.

    First, we fix some notation. Throughout this section, we will study multiple orthogonal polynomials on the domain
    \begin{equation}
        T = \{(x,y)\in\RR^2\,:\,x,y\geq 0,\, 0\leq x+y\leq 1\},
    \end{equation}
    which is the triangle whose vertices are $(0,0)$, $(1,0)$ and $(0,1)$, see Figure~\ref{fig:domains}(e). In $T$, we define the weights
    \begin{equation}
        \label{eq:weights-triangle}
            W_j(x,y)=x^{\alpha_j}y^{\beta_j}(1-x-y)^{\gamma},\qquad(x,y)\in T,\qquad j=1,\dots,r.
    \end{equation}
     
    Let $\vec n=(n_1,\dots,n_r),\vec k=(k_1,\dots,k_r)\in\NN_0^r$ be two multi-indices with $0\leq k_j\leq n_j$. The first study of MOPs on $T$ with respect to weights of the form \eqref{eq:weights-triangle} was given in \cite{FFMV26}, where polynomials given by the bivariate Rodrigues formula
    \begin{equation}\label{eq:rodrigues-triangle}
    \begin{split}
    	U^{(\vec \alpha,\vec \beta,\gamma)}_{\vec n,\vec k}(x,y)  = \dfrac{1}{W_r(x,y)} &\frac{\partial^{n_r}}{\partial x^{n_r-k_r}\partial y^{k_r}}\Big(x^{n_r-k_r+\alpha_r-\alpha_{r-1}}y^{k_r+\beta_r-\beta_{r-1}}  \cdots \times \\ & \times  \frac{\partial^{n_1}}{\partial x^{n_1-k_1}\partial y^{k_1}}(x^{n_1-k_1+\alpha_1}y^{k_1+\beta_1}(1-x-y)^{\gamma+n_1+\cdots+n_r})\Big),
    \end{split}
    \end{equation}
    were proven to be orthogonal to polynomials of degree less than $n_j$ with respect to the $j$-th weight. Now, we present a different family of MOPs  on the triangle using the Koornwinder construction, which will share this property with the polynomials $U^{(\vec \alpha,\vec \beta,\gamma)}_{\vec n,\vec k}(x,y)$ despite the differences in their construction.

    From now on, we assume that, for $i\ne j$,  either $\alpha_j-\alpha_i\notin\ZZ$ when $i\ne j$ or $\beta_j-\beta_i\notin\ZZ$.

    \subsubsection{MOPs on the Triangle via the Koornwinder construction}\label{ssec:koornwinder-triangle}
    
    Following Section~\ref{ssec:koornwinder-case-2}, we set
    \begin{equation*}
        \omega_j(x)=x^{\alpha_j}(1-x)^{\beta_j+\gamma}, \quad x\in[0,1],\qquad \sigma_j(x)=x^{\beta_j}(1-x)^\gamma,\quad x\in[0,1],\qquad j=1,\dots,r,
    \end{equation*}
    and $\rho(x)=1-x$. Thus, we need MOPs  with respect to the weights $x^{\alpha_j}(1-x)^{\beta_j+\gamma+|\vec k|+k_j+1}$, which were introduced in Section~\ref{subsubsec:jacobi_pineiro}, Case 3. We denote the corresponding Type~II MOP by $P^{(\vec\alpha,\widetilde\beta)}_{\vec n-\vec k}(x)$, where the components of $\widetilde\beta\in\RR^r$ are $(\widetilde\beta)_j=\beta_j+\gamma+|\vec k|+k_j+1$. On the other hand, MOPs  with respect to $\sigma_1,\dots,\sigma_r$ are the well-known Jacobi--Pi\~neiro polynomials $P^{(\vec \beta,\gamma)}_{\vec n}(x)$. 

    By \eqref{eq:koornwinder-multiple-2}, we define the polynomials
    \begin{equation}\label{eq:multiple-triangle-koornwinder}
    \mathbf P^{(\vec\alpha,\vec\beta,\gamma)}_{\vec n,\vec k}(x,y)=P^{(\vec\alpha,\widetilde\beta)}_{\vec n-\vec k}(x) \,(1-x)^{|\vec k|}\,P^{(\vec \beta,\gamma)}_{\vec k}\left(\dfrac{y}{1-x}\right).
    \end{equation}
    By Proposition~\ref{prop:orthogonality-2}, these polynomials are orthogonal to all the polynomials of total degree lower than $n_j$ and to those with degree lower than $k_j$ in the variable `$y$', with respect to $W_j$ defined in \eqref{eq:weights-triangle} on $T$. In particular, they are orthogonal to polynomials of total degree less than $n_j$ with respect to the $j$-th weight.

    This construction can be computed explicitly. However, little is known about MOPs  for weights of the type $x^{\alpha_j}(1-x)^{\beta_j}$. Consequently, computing the polynomials \eqref{eq:multiple-triangle-koornwinder}, while feasible, might be computationally costly as $|\vec n-\vec k|$ increases. Note that if we fix $\alpha_j=\alpha>-1$ for every $j$, then both polynomials in \eqref{eq:multiple-triangle-koornwinder} become Jacobi--Pi\~neiro polynomials $P^{(\alpha,\tilde\beta)}_{\vec n-\vec k}$ and $P^{(\vec \beta,\gamma)}_{\vec k}$, which are efficiently computable via the Rodrigues formula \eqref{eq:rodrigues-jp}. The resulting construction is orthogonal to polynomials of total degree lower than $n_j$ and lower than $k_j$ on the variable `$y$' with respect to the weight $x^\alpha\, y^{\beta_j}(1-x-y)^\gamma$ on $T$, $j=1,\dots,r$.

    \subsubsection{A Symmetric Construction of Multiple Orthogonal Polynomials on the Triangle}
    
    Recall $P^{(\vec\alpha,\vec\beta)}_{\vec n}$, the Type~II MOPs  with respect to the weights $x^{\alpha_j}(1-x)^{\beta_j}$, $j=1,\dots,r$ introduced in Section~\ref{subsubsec:jacobi_pineiro}, Case 3. Also, recall the Jacobi--Pi\~neiro polynomials $P^{(\gamma,\vec\alpha)}_{\vec n}$ and suppose that $\alpha_j+\beta_j-\alpha_i-\beta_i\notin\ZZ$ whenever $j\ne i$. Define
    \begin{equation}\label{eq:mop-triangle-J}
        J^{(\vec\alpha,\vec\beta,\gamma)}_{\vec n,\vec k}(x,y)=P^{(\gamma, \vec\beta_{\vec k})}_{\vec n-\vec k}(1-x-y) \,(x+y)^{|\vec k|}\, P^{(\vec\beta,\vec\alpha)}_{\vec k}\left(\dfrac{y}{x+y}\right),
    \end{equation}
    where $(\vec\beta_{\vec k})_j=|\vec k|+k_j+\alpha_j+\beta_j+1$, $j=1,\dots,r$.

    These polynomials are an extension of the so-called Jacobi polynomials on the triangle $J_{n,k}^{\alpha,\beta,\gamma}(x,y)$ introduced in \cite[Eq. (4.1)]{Xu17}, which are popular for exploiting the symmetry of the triangle between the ``variables'' $x$, $y$ and $1-x-y$. In this multiple case, let us show the orthogonality conditions satisfied by \eqref{eq:mop-triangle-J}.

    \begin{proposition}
        Let $\vec n=(n_1,\dots,n_r),\vec k=(k_1,\dots,k_r)\in\NN_0^r$ be two multi-indices with $0\leq k_j\leq n_j$ and consider the polynomial $J^{(\vec\alpha,\vec\beta,\gamma)}_{\vec n,\vec k}$ introduced in \eqref{eq:mop-triangle-J}. Then,
        \begin{equation}\label{eq:orthogonality-mop-triangle}
            \int_T J^{(\vec\alpha,\vec\beta,\gamma)}_{\vec n,\vec k}(x,y)\, x^t y^s \, W_j(x,y)\, \mathrm dx\,\mathrm dy =0 \qquad \text{ if } t+s<n_j,\qquad j=1,\dots,r,
        \end{equation}
        where the weight $W_j$ is given in \eqref{eq:weights-triangle}.
        
    \end{proposition}
    \begin{proof}
        Using the corresponding definitions, the integral to be computed is 
        $$
        I:=\int_T P^{(\gamma,\vec\beta_{\vec k})}_{\vec n-\vec k}(1-x-y) \,(x+y)^{|\vec k|}\, P^{(\vec\beta,\vec \alpha)}_{\vec k}\left(\dfrac{y}{x+y}\right)\, x^t y^s \,x^{\alpha_j}y^{\beta_j}(1-x-y)^\gamma\, \mathrm dx\,\mathrm dy. 
        $$
        We use the change of variables $z=1-x-y$, $w=\dfrac{y}{x+y}$, which leads to $x=(1-z)(1-w)$, $y=w(1-z)$, $x+y=1-z$ and $\mathrm dx\,\mathrm dy = (1-z)\,\mathrm d z\,\mathrm d w$, so that the integral becomes
        $$
        \int_0^1\int_0^1 P^{(\gamma,\vec\beta_{\vec k})}_{\vec n-\vec k}(z)\, (1-z)^{|\vec k|}\,P^{(\vec\beta,\vec \alpha)}_{\vec k}(w)\, z^\gamma\,(1-z)^{t+s+\alpha_j+\beta_j+1}\; w^{s+\beta_j}\,(1-w)^{t+\alpha_j}\mathrm dz\,\mathrm dw.
        $$
        This integral can be split into two univariate integrals: one, $I_z$ with respect to $z$ and another, $I_w$ with respect to $w$. Then, $I=I_z\times I_w=0$ if either $I_z=0$ or $I_w=0$, where
        \begin{equation*}
            \begin{aligned}
                I_z&=\int_0^1 P^{(\gamma,\vec\beta_{\vec k})}_{\vec n-\vec k}(z)\, (1-z)^{t+s} z^\gamma (1-z)^{|\vec k|+\alpha_j+\beta_j+1}\,\mathrm d z,\\
                I_w&= \int_0^1 P^{(\vec\beta,\vec \alpha)}_{\vec k}(w)\, w^s(1-w)^t\, w^{\beta_j}(1-w)^{\alpha_j}\,\mathrm d w.
            \end{aligned}
        \end{equation*}
        By Type~II multiple orthogonality of $P^{(\vec\beta,\vec \alpha)}_{\vec k}(w)$ \eqref{eq:jp_orth_both_vec}, we have that $I_w=0$ if $\deg(w^s(1-w)^t)=t+s<k_j$.

        Now assume that $t+s\geq k_j$, \textit{i.e.}, $t+s=k_j+\delta$ with $\delta\geq 0$. Then, $I_z$ reads
        $$
        I_z=\int_0^1 P^{(\gamma,\vec\beta_{\vec k})}_{\vec n-\vec k}(z)\, (1-z)^{\delta} z^\gamma (1-z)^{|\vec k|+k_j+\alpha_j+\beta_j+1}\,\mathrm d z = \int_0^1 P^{(\gamma,\vec\beta_{\vec k})}_{\vec n-\vec k}(z)\, (1-z)^{\delta} z^\gamma (1-z)^{(\vec\beta_{\vec k})_j}\,\mathrm d z.
        $$
        We recover the associated weight of $P^{(\gamma,\vec\beta_{\vec k})}_{\vec n-\vec k}(z)$, so that by \eqref{eq:jp_orth_beta_vec}  $I_z =0$ if $\deg((1-z)^\delta)=\delta<n_j-k_j$. Equivalently, $I_z=0$ (and in turn $I=0$) if $t+s<n_j$.

    \end{proof}

    As in the construction in Section~\ref{ssec:koornwinder-triangle}, although the construction of the polynomials $P^{(\vec\alpha,\vec\beta)}_{\vec n}$ is needed in general, it can be avoided if one of the free parameters $\alpha_j$ or $\beta_j$ is fixed. Then:

    \begin{itemize}
        \item Assume that $\alpha_1=\cdots=\alpha_r=\alpha>-1$ and $\beta_j-\beta_i\notin\ZZ$ if $i\neq j$. Then,
        $$
        J^{(\alpha,\vec\beta,\gamma)}_{\vec n,\vec k}(x,y)=P^{(\gamma,\vec\beta_{\vec k})}_{\vec n-\vec k}(1-x-y) \,(x+y)^{|\vec k|}\, P^{(\vec\beta, \alpha)}_{\vec k}\left(\dfrac{y}{x+y}\right),
        $$
        where $P^{(\gamma,\vec\beta_{\vec k})}_{\vec n-\vec k}$ and $P^{(\vec\beta, \alpha)}_{\vec k}$ denote the Jacobi--Pi\~neiro polynomials with $x\mapsto1-x$ and the Jacobi--Pi\~neiro polynomials respectively, satisfy \eqref{eq:orthogonality-mop-triangle} with the weights
        $$
        W_j(x,y)=x^\alpha\, y^{\beta_j}\, (1-x-y)^\gamma, \qquad j=1,\dots,r.
        $$
        
        \item  Assume that $\beta_1=\cdots=\beta_r=\beta>-1$ and $\alpha_j-\alpha_i\notin\ZZ$ if $i\neq j$. Then,
        $$
        J^{(\vec\alpha,\beta,\gamma)}_{\vec n,\vec k}(x,y)=P^{(\gamma,\vec\beta_{\vec k})}_{\vec n-\vec k}(1-x-y) \,(x+y)^{|\vec k|}\, P^{(\beta, \vec\alpha)}_{\vec k}\left(\dfrac{y}{x+y}\right),
        $$
        satisfy \eqref{eq:orthogonality-mop-triangle} with the weights
        $$
        W_j(x,y)=x^{\alpha_j}\, y^{\beta}\, (1-x-y)^\gamma, \qquad j=1,\dots,r.
        $$

        \item Assume $r=1$ and  $\alpha_1=\alpha>-1$ and $\beta_1=\beta>-1$. Define the polynomials
        \begin{equation}\label{eq:Jacobi-triangle-Xu}
            J^{(\alpha,\beta,\gamma)}_{n,k}(x,y)=\widetilde P^{(\gamma,2k+\alpha+\beta+1)}_{n-k}(1-x-y) \,(x+y)^{k}\, \widetilde P^{(\beta,\alpha)}_{k}\left(\dfrac{y}{x+y}\right),
        \end{equation}
        where $\widetilde P^{(\beta,\alpha)}_n$ denote the Jacobi polynomials shifted to the interval $[0,1]$, \textit{i.e.}, orthogonal with respect to the weight $x^\beta (1-x)^\alpha$ (see \eqref{eq:jacobi_orth_shifted}), introduced in Section~\ref{subsubsec:jacobi_standard}.

    \end{itemize}

    The first two points show that if only one parameter is allowed to vary in the weights \eqref{eq:weights-triangle} while keeping the other two fixed, the construction becomes simpler, since it reduces to Jacobi--Pi\~neiro polynomials, for which there are several efficient computational methods, \textit{e.g.}, the Rodrigues formula \eqref{eq:rodrigues-jp}. Nevertheless, this comes at the expense of the broader ability to compute MOPs  efficiently for \eqref{eq:weights-triangle} when both parameters vary.

    Regarding the last item, note that the case $r=1$ of this construction leads to the original construction employed by Xu in \cite{Xu17}. 
    
    In the remainder of this section, we show some properties inherited from the symmetry of the triangle $T$ and the construction \eqref{eq:mop-triangle-J} itself. Indeed, we will study the invariance of these polynomials under permutations of $(x,y,1-x-y)$ and $(\vec\alpha,\vec\beta,\gamma)$. With this purpose, we will use the notation $(a_j)$ to denote a vector $\vec a=(a_1,\dots,a_r)\in\RR^r$ using just the components. In this way, $\vec\beta_{\vec k}$ is to be denoted by $(|\vec k|+k_j+\alpha_j+\beta_j+1)$.

    By permuting $x\mapsto y$, $y\mapsto1-x-y$, $1-x-y\mapsto x$, we define the polynomial
    \begin{equation}\label{eq:MOP-simplex-L}
        L_{\vec n,\vec k}^{\vec\alpha,\vec\beta,\gamma}(x,y)=P_{\vec n-\vec k}^{(\vec\alpha, (|\vec k|+k_j+\beta_j+\gamma+1))}(x)\,(1-x)^{|\vec k|}\, P^{(\gamma,\vec\beta)}_{\vec k}\left(1-\dfrac{y}{1-x}\right).
    \end{equation}
    Then, this polynomial satisfies the same orthogonality relations as $J_{\vec n,\vec k}^{\vec\alpha,\vec\beta,\gamma}$:
    $$
    \int_T L_{\vec n,\vec k}^{\vec\alpha,\vec\beta,\gamma}(x,y)\, x^t y^s \, W_j(x,y)\,\mathrm dx\, \mathrm d y = 0\qquad\text{ if }t+s<n_j,\quad j=1,\dots,r,
    $$
    where $W_j$ is given in \eqref{eq:weights-triangle}. 

    Since $P^{(\gamma,\vec\beta)}_{\vec k}\left(1-\dfrac{y}{1-x}\right)=(-1)^{|\vec k|}P^{(\vec\beta,\gamma)}_{\vec k}\left(\dfrac{y}{1-x}\right)$ owing to the properties of Jacobi--Pi\~neiro polynomials (see Section~\ref{subsubsec:jacobi_pineiro}, Cases 1 and 2), the polynomial \eqref{eq:MOP-simplex-L} is essentially the same as \eqref{eq:multiple-triangle-koornwinder} up to a sign factor, \textit{i.e.} $ L_{\vec n,\vec k}^{\vec\alpha,\vec\beta,\gamma}(x,y)=(-1)^{|\vec k|}\,\mathbf P^{(\vec\alpha,\vec\beta,\gamma)}_{\vec n,\vec k}(x,y)$.

    Finally, by permuting $x\mapsto 1-x-y$, $y\mapsto x$, $1-x-y\mapsto y$ in \eqref{eq:mop-triangle-J}, we have that
    \begin{equation}
        K_{\vec n,\vec k}^{\vec\alpha,\vec\beta,\gamma}(x,y)=P_{\vec n-\vec k}^{(\vec\beta, (|\vec k|+k_j+\alpha_j+\gamma+1))}(y)\,(1-y)^{|\vec k|}\, P^{(\vec\alpha,\gamma)}_{\vec k}\left(\dfrac{x}{1-y}\right) 
    \end{equation}
    also satisfies
    $$
    \int_T K_{\vec n,\vec k}^{\vec\alpha,\vec\beta,\gamma}(x,y)\, x^t y^s \, W_j(x,y)\,\mathrm dx\, \mathrm d y = 0\qquad\text{ if }t+s<n_j,\quad j=1,\dots,r.
    $$
    In fact, this polynomial could be derived by applying the permutation $(x,y,1-x-y)\longmapsto(y,1-x-y,x)$ together with $(\vec\alpha,\vec\beta,\gamma)\longmapsto(\vec\beta,\gamma,\vec\alpha)$ to $J_{\vec n,\vec k}^{\vec\alpha,\vec\beta,\gamma}$, with the understanding that permuting the parameters may, depending on the situation, lead to different families of Jacobi--Pi\~neiro polynomials. Allowing for this slight abuse of notation, the polynomials $J_{\vec n,\vec k}^{\vec\alpha,\vec\beta,\gamma}$, $L_{\vec n,\vec k}^{\vec\alpha,\vec\beta,\gamma}$, and $K_{\vec n,\vec k}^{\vec\alpha,\vec\beta,\gamma}$ satisfy the following identities:

    \begin{equation}
        \begin{aligned}
            J_{\vec n,\vec k}^{\vec\alpha,\vec\beta,\gamma}(x,y) &= L_{\vec n,\vec k}^{\gamma,\vec\alpha,\vec\beta}(1-x-y,x) & J_{\vec n,\vec k}^{\vec\alpha,\vec\beta,\gamma}(x,y) &= K_{\vec n,\vec k}^{\vec\beta,\gamma,\vec\alpha}(y,1-x-y)
            \\
            L_{\vec n,\vec k}^{\vec\alpha,\vec\beta,\gamma}(x,y) &= J_{\vec n,\vec k}^{\vec\beta,\gamma,\vec\alpha}(y,1-x-y) & L_{\vec n,\vec k}^{\vec\alpha,\vec\beta,\gamma}(x,y) &= K_{\vec n,\vec k}^{\gamma,\vec\alpha,\vec\beta}(1-x-y,x)\\
            K_{\vec n,\vec k}^{\vec\alpha,\vec\beta,\gamma}(x,y) &= J_{\vec n,\vec k}^{\gamma,\vec\alpha,\vec\beta}(1-x-y,x) & K_{\vec n,\vec k}^{\vec\alpha,\vec\beta,\gamma}(x,y) &= L_{\vec n,\vec k}^{\vec\beta,\gamma,\vec\alpha}(y,1-x-y)\\
        \end{aligned}
    \end{equation}\\
   
    To conclude this section, in which several examples of MOPs  on planar domains have been presented, we point out an important observation. While Koornwinder’s method for standard OPs provides a way to build OPs on the unit disk $D=\{(x,y)\in\RR^2: x^2+y^2\leq 1\}$ (see \cite[Section~3.3, Class~II]{Koo75} or \cite[Section~2.6.1]{DX14}), that definition cannot be extended to the multiple orthogonal setting described in Section~\ref{sec:koornwinder}. Indeed, the construction of OPs on the disk relies on combining two Gegenbauer polynomials; however, as discussed in Section~\ref{subsubsec:gegenbauer_standard} and in view of Proposition~\ref{prop:MOP-symmetric}, one cannot obtain a perfect system for Gegenbauer-type weights \eqref{eq:gegenbauer_weight} when the parameter is allowed to vary. Consequently, even though it is in principle possible to define MOPs  on the disk through the Koornwinder method, it would not be feasible to compute the necessary polynomials for all $\vec n\geq \vec k$, which restricts the potential usefulness of this approach. Nevertheless, multiple orthogonal polynomials on the unit disk --- along with extensions to the $d$-dimensional ball and other radial generalizations --- have been studied in \cite{FV26b}.

    \section{Multiple Orthogonal Polynomials on the Simplex}\label{sec:MOP-simplex}

    In the previous sections, we have discussed standard and multiple orthogonal polynomials on several planar (2-dimensional) domains. This is because the Koornwinder construction is inherently a bivariate method. Nevertheless, in the case of the Triangle (see Section~\ref{ssec:koornwinder-triangle}), the construction is extensible to a higher number of variables \cite[Section~5.3]{DX14}. In this section, we will show the natural extension to multiple orthogonality of the classical OP on the so-called `Simplex' via the Koornwinder method.

    As a $d$-variate extension of the triangle $T$, consider the domain
    $$
    T^d = \{\mathbf x =(x_1,\dots,x_d)\in\RR^d\, : \, x_1,\dots,x_d\geq 0,\, 0\leq x_1+\cdots+x_d\leq 1\}.
    $$
    
    For the sake of simplicity, only the case $r=2$ will be studied, although this construction naturally extends to an arbitrary number of weights.

    Consider $\mathbf x =(x_1,\dots,x_d)\in T^d$. If $\mathbf a =(a_1,\dots,a_d)\in\NN_0^d$, then we will denote a $d$-dimensional monomial as $(\mathbf x)^{\mathbf a}=x_1^{a_1}\cdots x_d^{a_d}$. Also, if $\mathbf u=(u_1,\dots,u_d)\in\RR^d$, we will use the $\ell^1$ norm of a vector $|\mathbf u|=u_1+\cdots+u_d$ and the truncations $\mathbf u^i:=(u_i,\dots,u_d)\in\RR^{d-i+1}$, $\mathbf u_i:=(u_1,\dots,u_i)\in\RR^i$, so that $(\mathbf x_i)^{\mathbf a_i}=x_1^{a_1}\cdots x_i^{a_i}$. 

    Let $\alpha_1,\dots,\alpha_d,\alpha_{d+1}>-1$, $\beta_1,\dots,\beta_d,\beta_{d+1}>-1$ and denote $\vec\alpha=(\alpha_1,\dots,\alpha_{d+1})\in\RR^{d+1}$, $\vec\beta=(\beta_1,\dots,\beta_{d+1})\in\RR^{d+1}$. Using the previous notation $\vec\alpha^i = (\alpha_i,\dots,\alpha_{d+1})$ and $\vec\alpha_i = (\alpha_1,\dots,\alpha_{i})$. We define the following weights
    \begin{equation}\label{eq:weights-simplex-d}
    \begin{aligned}
        W_{\vec\alpha}(\mathbf x)&=(\mathbf x)^{\vec\alpha_d}(1-|\mathbf x|)^{\alpha_{d+1}} =x_1^{\alpha_1}\cdots x_d^{\alpha_d}(1-|\mathbf x|)^{\alpha_{d+1}},\\
        W_{\vec\beta}(\mathbf x)&=(\mathbf x)^{\vec\beta_d}(1-|\mathbf x|)^{\beta_{d+1}} = x_1^{\beta_1}\cdots x_d^{\beta_d}(1-|\mathbf x|)^{\beta_{d+1}},
    \end{aligned}
    \end{equation}
    both supported on $T^d$. For $\mathbf n = (n_1,\dots,n_d), \mathbf m =(m_1,\dots,m_d)\in\NN_0^d$, define the polynomial
    \begin{equation}
        \label{eq:polynomial-simplex-d}
        P_{\mathbf n,\mathbf m}^{\vec \alpha,\vec\beta}(\mathbf x)=\prod_{i=1}^d(1-|\mathbf x_{i-1}|)^{n_i+m_i}\, P_{(n_i,m_i)}^{(\vec \gamma_i,\vec \lambda_i)}\left(\dfrac{x_i}{1-|\mathbf x_{i-1}|}\right),
    \end{equation}
    where the parameters of the univariate MOPs  -- introduced in Section~\ref{subsubsec:jacobi_pineiro}, Case 3 -- are 
    $$
    \vec\gamma_i=(\alpha_i,\beta_i)\quad\text{ and }\quad\vec\lambda_i =(2|\mathbf n^{i+1}|+|\mathbf m^{i+1}|+|\vec\alpha^{i+1}|+d-i,|\mathbf n^{i+1}|+2|\mathbf m^{i+1}|+|\vec\beta^{i+1}|+d-i),
    $$
    where $|\mathbf n^{d+1}|=|\mathbf m^{d+1}|=0$. In order to get the AT condition for the univariate MOPs  involved in \eqref{eq:polynomial-simplex-d}, assume that $|\vec\alpha^i|-|\vec\beta^i|\notin\ZZ$ for $i=1,\dots,d$. 
    Observe that $P_{\mathbf n,\mathbf m}^{\vec \alpha,\vec\beta}$ is a $d$-variate polynomial of total degree $|\mathbf n|+|\mathbf m|$.  Moreover, these polynomials satisfy the following orthogonality relations. 

    \begin{proposition}
        Let $\mathbf n = (n_1,\dots,n_d), \mathbf m =(m_1,\dots,m_d)\in\NN_0^d$ and consider the polynomial $P_{\mathbf n,\mathbf m}^{\vec \alpha,\vec\beta}(\mathbf x)$ introduced in \eqref{eq:polynomial-simplex-d}. The following orthogonality relations are satisfied: 

        \begin{equation}\label{eq:orthogonality-simplex-d-alpha}
            \int_{T^d} P_{\mathbf n,\mathbf m}^{\vec \alpha,\vec\beta}(\mathbf x)\, (\mathbf x)^{\mathbf a}\, W_{\vec\alpha}(\mathbf x)\,\mathrm d\mathbf x = 0\quad \text{ if } |\mathbf a^i| <|\mathbf n^i|\text{ for some }i\in\{1,\dots,d\}. 
        \end{equation}
        \begin{equation}\label{eq:orthogonality-simplex-d-beta}
            \int_{T^d} P_{\mathbf n,\mathbf m}^{\vec \alpha,\vec\beta}(\mathbf x)\, (\mathbf x)^{\mathbf a}\, W_{\vec\beta}(\mathbf x)\,\mathrm d\mathrm x = 0\quad \text{ if } |\mathbf a^i| <|\mathbf m^i|\text{ for some }i\in\{1,\dots,d\}. 
        \end{equation}
    In particular, these conditions hold whenever $|\mathbf a|<|\mathbf n|$ and $|\mathbf a|<|\mathbf m|$ respectively.
    \end{proposition}
    \begin{proof}
    We focus on \eqref{eq:orthogonality-simplex-d-alpha}, as the proof of \eqref{eq:orthogonality-simplex-d-beta} is analogous. Using the definitions of the polynomial and the weight, we have
    $$
    I:= \int_{T^d} \prod_{i=1}^d(1-|\mathbf x_{i-1}|)^{n_i+m_i}\, P_{(n_i,m_i)}^{(\vec \gamma_i,\vec \lambda_i)}\left(\dfrac{x_i}{1-|\mathbf x_{i-1}|}\right)\,  (\mathbf x)^{\mathbf a+\vec \alpha_d}\,(1-|\mathbf x|)^{\alpha_{d+1}}\,\mathrm d\mathrm x.
    $$
    Applying the change of variables $x_d = y(1-|\mathbf x_{d-1}|)$, $\mathrm d x_d =(1-|\mathbf x_{d-1}|)\,\mathrm d y$, the integral reads
    \begin{multline*}
        \int_{T^{d-1}}\int_0^1 \prod_{i=1}^{d-1}(1-|\mathbf x_{i-1}|)^{n_i+m_i}\, P_{(n_i,m_i)}^{(\vec \gamma_i,\vec \lambda_i)}\left(\dfrac{x_i}{1-|\mathbf x_{i-1}|}\right)\, (1-|\mathbf x_{d-1}|)^{n_d+m_d} P_{(n_d,m_d)}^{(\vec \gamma_d,\vec \lambda_d)}(y)\\ \times  (\mathbf x_{d-1})^{\mathbf a_{d-1}+\vec\alpha_{d-1}}y^{a_d+\alpha_d}(1-|\mathbf x_{d-1}|)^{a_d+\alpha_d}(1-y)^{\alpha_{d+1}}(1-|\mathbf x_{d-1}|)^{\alpha_{d+1}+1}\, \mathrm d\mathbf x_{d-1}\,\mathrm dy ,
    \end{multline*}
    where we used the equality $(1-|\mathbf x_{i}|)^\ell=\left(1-\dfrac{x_i}{1-|\mathbf x_{i-1}|}\right)^\ell (1-|\mathbf x_{i-1}|)^\ell$ for any $i$ and $\ell\in\mathbb R$.

    Then, it is possible to split the integral as the product of an integral on $T^{d-1}$ and an integral in $y$:
    \begin{equation}\label{eq:integrals-orthogonality-simplex-d}
        \begin{aligned}
            I = &\int_{T^{d-1}}\prod_{i=1}^{d-1}(1-|\mathbf x_{i-1}|)^{n_i+m_i} P_{(n_i,m_i)}^{(\vec \gamma_i,\vec \lambda_i)}\left(\dfrac{x_i}{1-|\mathbf x_{i-1}|}\right)\\ &\times (\mathbf x_{d-1})^{\mathbf a_{d-1}+\vec\alpha_{d-1}} \, (1-|\mathbf x_{d-1}|)^{n_d+m_d+a_d+\alpha_d+\alpha_{d+1}+1} \,\mathrm d\mathbf x_{d-1}\\
             \times & \int_0^1 P_{(n_d,m_d)}^{(\vec \gamma_d,\vec \lambda_d)}(y)\, y^{a_d}\, (y^{\alpha_d}(1-y)^{\alpha_{d+1}}) \, \mathrm dy.
        \end{aligned}
    \end{equation}
    As $\vec\gamma_d=(\alpha_d,\beta_d)$ and $\vec\lambda_d =(\alpha_{d+1},\beta_{d+1})$ and we are working with $W_{\vec\alpha}$, the first weight in \eqref{eq:weights-simplex-d}, the second integral vanishes if $a_d<n_d$, equivalently, if $|\mathbf a^d|<|\mathbf n^d|$.

    Now assume that $a_d = n_d +\delta_d$ with $\delta_d\in\NN_0$. If we focus on the first integral of \eqref{eq:integrals-orthogonality-simplex-d}, the one on $T^{d-1}$, it is possible to use the change of variables $x_{d-1}=y(1-|\mathbf x_{d-2}|)$ to get
    \begin{equation}\label{eq:integrals-orthogonality-simplex-d-1}
        \begin{aligned}
            &\int_{T^{d-2}}\prod_{i=1}^{d-2}(1-|\mathbf x_{i-1}|)^{n_i+m_i} P_{(n_i,m_i)}^{(\vec \gamma_i,\vec \lambda_i)}\left(\dfrac{x_i}{1-|\mathbf x_{i-1}|}\right)\\ &\times (\mathbf x_{d-2})^{\mathbf a_{d-2}+\vec\alpha_{d-2}} \, (1-|\mathbf x_{d-2}|)^{|\mathbf n^{d-1}|+|\mathbf m^{d-1}|+|\mathbf a^{d-1}|+|\vec \alpha^{d-1}|+2} \,\mathrm d\mathbf x_{d-2}\\
             & \times  \int_0^1 P_{(n_{d-1},m_{d-1})}^{(\vec \gamma_{d-1},\vec \lambda_{d-1})}(y)\, y^{a_{d-1}}(1-y)^{a_d}\, (y^{\alpha_{d-1}}(1-y)^{n_d+m_d+\alpha_d+\alpha_{d+1}+1}) \, \mathrm dy.
        \end{aligned}
    \end{equation}
    Since we are assuming $a_d = n_d +\delta_d$, then the second integral in \eqref{eq:integrals-orthogonality-simplex-d-1} is
    \begin{equation}\label{eq:integrals-orthogonality-simplex-d-2}
        \int_0^1 P_{(n_{d-1},m_{d-1})}^{(\vec \gamma_{d-1},\vec \lambda_{d-1})}(y)\, y^{a_{d-1}}(1-y)^{\delta_d}\, (y^{\alpha_{d-1}}(1-y)^{2n_d+m_d+\alpha_d+\alpha_{d+1}+1}) \, \mathrm dy.
    \end{equation}
    Recall $\vec \gamma_{d-1}=(\alpha_{d-1},\beta_{d-1})$ and $\vec \lambda_{d-1}=(2n_d+m_d+\alpha_d+\alpha_{d+1}+1,n_d+2m_d+\beta_d+\beta_{d+1}+1)$, so that \eqref{eq:integrals-orthogonality-simplex-d-2} vanishes if $\deg( y^{a_{d-1}}(1-y)^{\delta_d})=a_{d-1}+\delta_d<n_{d-1}$, that is, if $a_{d-1}+a_d=|\mathbf a^{d-1}|<n_{d-1}+n_d=|\mathbf n^{d-1}|$.

    Repeating this process $d$ times, we finally get that $I=0$ if $|\mathbf a^{i}|<|\mathbf n^{i}|$ for some $i$.

    \end{proof}

    \section{Multiple Orthogonal Polynomials on the Cone}\label{sec:MOP-cone}

    In Section~\ref{ssec:MOP-wedge}, we introduced MOPs  on the domain known as the wedge $\mathbb V^2$. Indeed, as mentioned, the wedge is the two-dimensional version of the cone, which, using in $\RR^{d+1}$ the notation $(\mathbf x, t)=(x_1,\dots,x_d,t)\in\RR^{d+1}$, is defined as
    \begin{equation}
        \label{eq:cone}
        \mathbb V^{d+1} =\{(\mathbf x,t)\in\RR^{d+1}\,:\, \|\mathbf x\|\leq t,\, \mathbf x\in\RR^{d},\, t\geq 0\}.
    \end{equation}
    In this multivariate case, it is possible to proceed in a similar way, so that using an appropriate change of variables, $\mathbb V^{d+1}$ is mapped to $[0,+\infty)\times B^d$. Thus, combining multiple Laguerre polynomials and OPs on the $d$-dimensional unit ball it is possible to define MOPs  on the cone, following the approach described in \cite{ABFMX23}. 

    Let $\mu>-1$ and $\beta_1,\dots,\beta_r >1-d-2\mu$ be such that $\beta_j-\beta_i\notin\ZZ$ if $j\neq i$  and denote by $\vec\beta =(\beta_1,\dots,\beta_r)$. We define the following weights supported on the cone $\mathbb V^{d+1}$.
    \begin{equation}\label{eq:weights-cone}
        W_j(\mathbf x,t)=W^{(\beta_j,\mu)}(\mathbf x, t)= t^{\beta_j}\, e^{-t}\, (t^2-\|\mathbf x\|^2)^\mu,\qquad (\mathbf x,t)\in \mathbb V^{d+1}.
    \end{equation}
    
    In order to define MOPs  with respect to the weights \eqref{eq:weights-cone}, let $\{P_{\mathbf m}(\mathbf x)\, :\, |\mathbf m| = n\}$ denote a basis of $\mathcal V_n(B^d; W^{(\mu)})$, the space of orthogonal polynomials of degree $n$ with respect to the weight function $W^{(\mu)}(\mathbf x)=(1-\|\mathbf x\|^2)^\mu$, $\mu>-1$, on the $d$-dimensional ball $B^d=\{\mathbf x\in\RR^d\,:\, \|\mathbf x\|\le 1\}$. Consider the Laguerre MOPs  $L^{\vec\alpha}_{\vec n}(t)$ introduced in Section~\ref{subsubsec:multiple_laguerre}. Let $\vec n,\vec k\in\NN_0^r$ such that $0\leq k_j\leq n_j$ and $\{P_{\mathbf m}\,:\,|\mathbf m|=|\vec k|\}$ a basis of $\mathcal V_{|\vec k|}(B^d;W^{(\mu)})$. Based on the construction in \cite[Eq.~(2.12)]{ABFMX23}, we define
    \begin{equation}\label{eq:MOP-cone}
    \begin{aligned}
        \text{\textbf{Type II:}}\quad &\mathbf L^{(\vec\beta,\mu)}_{\vec n,\vec k,\mathbf m}(\mathbf x,t) = L_{\vec n-\vec k}^{(\widetilde\beta)}(t)\, t^{|\vec k|}\, P_{\mathbf m}\left(\dfrac{\mathbf x}{t}\right),\qquad |\mathbf m|=|\vec k|,\\
        \text{\textbf{Type I:}}\quad&\mathbf B^{(\vec\beta,\mu)}_{\vec n,\vec k,\mathbf m,(j)}(\mathbf x,t) = B_{\vec n-\vec k,j}^{(\widetilde\beta)}(t)\, t^{|\vec k|}\, P_{\mathbf m}\left(\dfrac{\mathbf x}{t}\right),\quad j=1,\dots,r, \qquad |\mathbf m|=|\vec k|,\\
    \end{aligned}
    \end{equation}
    where $(\widetilde \beta)_j=\beta_j+2\mu+2|\vec k|+d$, $j=1,\dots,r$.
    
    Now, we show that these polynomials satisfy multiple orthogonal relations with respect to $W_1,\dots,W_r$ defined in \eqref{eq:weights-cone}.

    \begin{proposition}
        Given $\vec n,\vec k\in\NN_0^r$ such that $0\leq k_j\leq n_j$ and $\mathbf m\in\NN_0^d$ with $|\mathbf m|=|\vec k|$, consider the polynomials $\mathbf L^{(\vec\beta,\mu)}_{\vec n,\vec k,\mathbf m}(\mathbf x,t)$ and  $\mathbf B^{(\vec\beta,\mu)}_{\vec n,\vec k,\mathbf m,(j)}(\mathbf x,t)$, $j=1,\dots,r$. These polynomials satisfy the following multiple orthogonal relations:
        \begin{itemize}
            \item Type II multiple orthogonality:
            \begin{equation}
                \label{eq:MOP-cone-type-ii-orth}
                \int_{\mathbb V^{d+1}} \mathbf L^{(\vec\beta,\mu)}_{\vec n,\vec k,\mathbf m}(\mathbf x,t)\, (\mathbf x)^{\mathbf a_d}t^{a_{d+1}} \, W_j(\mathbf x,t)\, \mathrm d\mathbf x\,\mathrm d t = 0 \qquad\text{ if }|\mathbf a|<n_j+|\vec k|-k_j,\quad j=1,\dots,r.
            \end{equation}
            \item Type I multiple orthogonality:
            \begin{equation}
                \label{eq:MOP-cone-type-i-orth}
                \sum_{j=1}^r\int_{\mathbb V^{d+1}} \mathbf B^{(\vec\beta,\mu)}_{\vec n,\vec k,\mathbf m,(j)}(\mathbf x,t)\, (\mathbf x)^{\mathbf a_d}t^{a_{d+1}} \, W_j(\mathbf x,t)\, \mathrm d\mathbf x\,\mathrm d t = 0 \qquad\text{ if }|\mathbf a|<|\vec n|-1.
            \end{equation}
        \end{itemize}
        Moreover, both \eqref{eq:MOP-cone-type-ii-orth} and \eqref{eq:MOP-cone-type-i-orth} hold whenever $|\mathbf a_d|<|\vec k|$.
    \end{proposition}
    \begin{proof}
        We start by proving \eqref{eq:MOP-cone-type-ii-orth}. Using the corresponding definitions, the integral is
        \begin{equation*}
            \int_{\mathbb V^{d+1}} L_{\vec n-\vec k}^{(\widetilde\beta)}(t)\, t^{|\vec k|}\, P_{\mathbf m}\left(\dfrac{\mathbf x}{t}\right)\, (\mathbf x)^{\mathbf a_d}t^{a_{d+1}} \,t^{\beta_j}\, e^{-t}(t^2-\|\mathbf x\|^2)^\mu\, \mathrm d\mathbf x\,\mathrm d t. 
        \end{equation*}
        Using the change of variables $\mathbf x = t\,\mathbf z$, $\mathrm d\mathbf x = t^d\,\mathrm d\mathbf z$, which leads to $\|\mathbf x\|^2 = t^2\|\mathbf z\|^2$, we rewrite the integral as
        $$
        \int_0^\infty \int_{B^d} L_{\vec n-\vec k}^{(\widetilde\beta)}(t)\, t^{|\vec k|}\, P_{\mathbf m}(\mathbf z) \,t^{|\mathbf a_d|}(\mathbf z)^{\mathbf a_d} t^{a_{d+1}}\, t^{\beta_j}e^{-t}t^{2\mu}(1-\|\mathbf z\|^2)^\mu\, t^d \, \mathrm d\mathbf z\,\mathrm d t.
        $$
        This integral can be written as the product of two integrals $I_t\times I_{\mathbf z}$, where
        \begin{equation*}
            \begin{aligned}
                I_t &= \int_0^\infty  L_{\vec n-\vec k}^{(\widetilde\beta)}(t)\,t^{|\mathbf a|}t^{\beta_j+2\mu+|\vec k|+d}e^{-t}\,\mathrm d t,\\
                I_{\mathbf z}&=\int_{B^d}P_{\mathbf m}(\mathbf z) (\mathbf z)^{\mathbf a_d}(1-\|\mathbf z\|^2)^\mu \, \mathrm d\mathbf z.
            \end{aligned}
        \end{equation*}
        Using the orthogonality of $P_{\mathbf m}$, we get that $I_{\mathbf z}=0$ if $|\mathbf a_d|<|\mathbf m|=|\vec k|$. Now assume that $|\mathbf a_d| = |\vec k|+\delta$. In that case, $I_t$ reads
        $$
        I_t = \int_0^\infty  L_{\vec n-\vec k}^{(\widetilde\beta)}(t)\,t^{a_{d+1}+\delta}t^{\beta_j+2\mu+2|\vec k|+d}e^{-t}\,\mathrm d t,
        $$
        and using the Type~II multiple orthogonality of multiple Laguerre polynomials \eqref{eq:laguerre_orth}, we deduce that $I_t=0$ whenever $a_{d+1}+\delta<n_j-k_j$, or equivalently, when $|\mathbf a|<n_j+|\vec k|-k_j$.

        The proof of \eqref{eq:MOP-cone-type-i-orth} is analogous. One only needs to replace the Type~II multiple orthogonality relations of $L_{\vec n-\vec k}^{(\widetilde\beta)}(t)$ by the corresponding Type~I relations satisfied by $B_{\vec n-\vec k,j}^{(\widetilde\beta)}(t)$.
        
    \end{proof}

    \section{Conclusions and future work}\label{sec:conclusions}

    With the methodology introduced in Section~\ref{sec:koornwinder}, we provide a systematic, explicit way of defining and computing MOPs  on planar domains, filling a key gap in the development of multiple orthogonal polynomials in the multivariate setting. This construction relies on the properties of standard and multiple families of OPs, leveraging their orthogonality relations to yield a consistent bivariate construction. Moreover, in the method of Section~\ref{ssec:koornwinder-case-1} (Case 1), we have introduced bivariate Type~I MOPs  and proved a biorthogonality relation, bringing the Type~I--Type~II duality  of MOPs  closer to the univariate foundations. 

    With these formulations, MOPs  on different kinds of domains were established. In the case of the triangle $T$, several possibilities for MOPs  combining univariate polynomials beyond the one in terms of Koornwinder's construction have been provided, and the relationships among them have been established. Nevertheless, recall that, whereas univariate OPs are unique up to a multiplicative constant, bivariate OPs are not unique, and neither are MOPs. Thus, some other approaches and definitions of MOPs  on these and other domains remain to be developed. 

    Finally, we explored the $d$-variate setting by providing expressions of MOPs  on the $d$-dimensional simplex $T^d$, which complements the  Rodrigues-type formula proposed in \cite{FFMV26}; and the first definition of MOPs  on the cone $\mathbb V^d$.

    As future work, we plan to explore applications or differential properties of polynomials of this kind, using the potential separation of variables given by the construction itself.

	\section*{Acknowledgements}

    The work is partially supported by grants PID2023-149117NB-I00, PID2024-155133NB-I00 and CEX 2020-001105-M, all funded by ``Ministerio de Ciencia, Innovación y Universidades'' 
    
    \noindent(MICIU/AEI/10.13039/501100011033 and ERDF/EU), Spain.

%	\newpage
	\bibliographystyle{plain}
    \bibliography{refs}

\end{document}